\documentclass[11pt]{amsart}
\usepackage{amssymb,amsmath,amsthm,mathrsfs}
\usepackage{graphicx}

\usepackage{a4wide}

\numberwithin{equation}{subsection}

\newtheorem{theorem}{Theorem}
\newtheorem{prop}{Proposition}[subsection]
\newtheorem{lemma}[prop]{Lemma}
\newtheorem{corollary}[prop]{Corollary}

\newcommand{\Ext}{\textrm{Ext}}
\newcommand{\Hom}{\textrm{Hom}}
\newcommand{\oc}{\mathcal O}
\newcommand{\gr}{\textrm{gr}}
\newcommand{\id}{\textrm{id}}
\newcommand{\Z}{\mathbb{Z}}
\newcommand{\N}{\mathbb{N}}
\newcommand{\im}{\textrm{im }}
\newcommand{\C}{\mathbb{C}}
\newcommand{\Der}{\operatorname{Der}}

\long\def\comment#1{}

\begin{document}
\title{The A-infinity algebra of an elliptic curve and the j-invariant}
\author{Robert Fisette}
\address{Department of Mathematics, University of Oregon, Eugene, OR  97403}
\email{rfisette@uoregon.edu}
\thanks{Written in part while hosted at the Institut Des Hautes \'{E}tudes Scientifiques}


\begin{abstract} We compute reduced Hochschild cohomology of $B=\Ext^*(\oc\oplus L,\oc\oplus L)$, where $\oc$ is the structure sheaf of an elliptic curve and $L$ is a line bundle of degree 1.  The result suggests an $A$-infinity equivalence between the $A$-infinity structure computed for $B$ in \cite{Pol-E} and a structure in which $m_3=m_4=m_5=0$.  We see that the $j$-invariant of the curve is explicitly related to this new structure. \end{abstract}
\maketitle

\section*{Introduction} Let $C=\C/\langle \Z\oplus\tau\Z\rangle$ be an elliptic curve, $D^b(C)$ the bounded derived category of coherent sheaves on $C$.  It is well-known that the category $D^b(C)$ is controlled by a natural $A$-infinity structure on the algebra $B=\Ext^*(\oc\oplus L,\oc\oplus L)$, where $\oc$ is the structure sheaf of $C$ and $L$ is a line bundle of degree 1.  \\

The interest in this structure was motivated by questions of homological mirror symmetry.  To determine this structure, one first chooses a resolution which calculates Ext.  The total complex of the resolution has the structure of a dg-algebra $A$, which we consider as an $A$-infinity algebra with $m_n=0$ for $n\geq 3.$  According to the theorem of Kadeishvili (\cite{Kad},\cite{Kel}), $B=H^*(A)$ has an $A$-infinity structure with $m_1=0$, $m_2$ induced by $m_2^A$, such that $A$ and $B$ are equivalent as $A$-infinity algebras.  This structure is unique up to strict $A$-infinity isomorphism. \\

Since $G=\oc\oplus L$ generates $D^b(C)$ as a triangulated category, $D^b(C)$ is equivalent to the derived category of perfect $A$-infinity modules over $B=\Ext^*(G,G)$.  For elliptic curves, it is also known that $D^b(C_1)\simeq D^b(C_2)$ if and only if $C_1\simeq C_2$; so this $A$-infinity structure on $B$ determines $C$ uniquely.  In \cite{Pol-E}, Polishchuk uses the Dolbeault complex and a formula of Merkulov from \cite{Merk} to compute an $A$-infinity structure on $B$ in terms of the Eisenstein series of the curve. \\

By choosing a generator in the same way for all curves, the answer in \cite{Pol-E} provides a family of $A$-infinity structures on the associative algebra $B$.  Thus it is of interest to consider $B$ as an associative algebra and examine its extensions to an $A$-infinity algebra with $m_1=0$ and $m_2=m_2^B.$  These extensions and their equivalences are governed by certain components of the Hochschild cohomology of $B$. \\

If we let $C$ be an elliptic curve over an algebraically closed field $k$, $P$ a closed $k$-point, then the multiplicative structure of $\Ext^*(\oc\oplus \oc(P),\oc\oplus\oc(P))$ is unchanged from the case where $k=\C$.  Thus it is valuable to compute Hochschild cohomology of this algebra for as general $k$ as possible, despite the fact that the computations in \cite{Pol-E} were done only for a complex curve. \\

The main goal of this paper is to compute the reduced Hochschild cohomology of $B$ as an associative algebra over $R=k\langle \id_{\oc},\id_L\rangle$, where $k$ is a field with $\operatorname{char} k\neq 2,3$.  We use the reduced complex (as in \cite{Seidel} (20d)) as it is known to be quasi-isomorphic to the full complex for a unital algebra.  The specific relationship of this cohomology to $A_\infty$-structures is well-known and recalled here in Section~\ref{HCAinf}.  \\

The fact that there are no nontrivial cocycles with internal degree -2 in $HH^4(B)$ implies the existence of a strict $A$-infinity isomorphism between the $A$-infinity algebra computed in \cite{Pol-E} with a structure in which $m_3=m_4=m_5=0$.  In this new structure, we see that the $j$-invariant of the curve is explicitly related to $m_6$ and $m_8$, thus making the dependence of this $A$-infinity structure on the isomorphism class of the curve completely explicit. \\

\section{Notation, conventions, and recollections}
\label{conventions}

When $f:A\to B$ is a map of graded objects, we say that $f$ is homogeneous of internal degree $n$ if $\deg f(x)=\deg(x)+n$ for all $x\in A$.  If $C^\bullet$ is a cochain complex with a differential which preserves internal degree, then $C^\bullet_{(n)}$ will mean the subcomplex of maps of internal degree $n$, and similarly $HH^\bullet_{(n)}$ refers to the cohomology of such a subcomplex.

When $A$ is a graded algebra, we define the tensor algebra $T(A)=\oplus_{i=1}^\infty A^{\otimes i}$.  $T(A)$ inherits an internal grading from $A$ by $$|x_1\otimes\cdots\otimes x_n|=\sum_{i=1}^n|x_i.$$  We let $S$ be the functor on a category of graded objects that shifts grading by -1, i.e. $(SA)_i=A_{i+1}$.  When $V$ is a graded space, $f:T(V)\to T(V)$ a homogeneous map with respect to the internal grading, we recall that $f$ is a superderivation if for homogeneous $v_1,v_2\in T(V)$ we have $$f(v_1\otimes v_2)=f(v_1)\otimes v_2+(-1)^{|v_1||f|}v_1\otimes f(v_2).$$  We have a graded Lie superalgebra $$\Der T(V)=\oplus_{l\in\Z} \Der_l T(V)$$ of superderivations on $T(V)$, where $\Der_lT(V)$ are the homogeneous superderivations of internal degree $l$.  Let $d\in \Der_i T(V), \delta\in\Der_jT(V),$ and $D\in\Der_k T(V)$.  Then the bracket is defined by $$[d,\delta]=d\delta-(-1)^{ij}\delta d.$$  In particular if $i,j,k$ are odd (as they will be in our case), then $$[d,\delta]=d\delta+\delta d$$ and $$[d,[\delta,D]]+[\delta,[D,d]]+[D,[d,\delta]]=0.$$

\subsection{Eisenstein series}
\label{Eisseries}

For a lattice $\Lambda\subset \C$ with basis $\omega_1,\omega_2$, with $m,n$ integers of the same parity, we set
$$f_{m,n}(\Lambda)=\left(\dfrac{\pi}{a(\Lambda)}\right)^m\sum_{\omega\in\Lambda/\{0\}}\dfrac{\bar{\omega}^m}{\omega^n}\operatorname{exp}\left(-\dfrac{\pi}{a(\Lambda)}|\omega|^2\right),$$ where $a(\Lambda)=\operatorname{Im}(\bar{\omega_1}\omega_2)$ is the area of $\C/\Lambda$.  Then for integers $a,b\geq 0$ of different parity we set $$g_{a,b}(\Lambda)=\sum_{k\geq 0}k!\left(\binom{a}{k}+\binom{b}{k}\right)f_{a+b-k,k+1}(\Lambda).$$  When $m,n$ are of different parity or $a,b$ the same parity, then $f_{m,n}(\Lambda)=g_{a,b}(\Lambda)=0$.  It is shown in \cite{Pol-E} that $g_{a,b}$ is a polynomial in $e_2^*,e_4,\ldots,e_{a+b+1}$ with rational coefficients, where $e_{2k}$ are the Eisenstein series for $\Lambda$ defined as $$e_{2k}(\Lambda)=\sum_{\omega\in\Lambda/\{0\}}\dfrac{1}{\omega^{2k}}$$ for $k\geq 2$, and $$e_2(\Lambda)=\sum_m\sum_{n; n\neq 0 \ \textrm{if} \ m=0}\dfrac{1}{(m\omega_2+n\omega_1)^2},$$ and $$e_2^*(\Lambda)=e_2(\omega_1,\omega_2)-\dfrac{\pi}{a(L)}\dfrac{\bar{\omega_1}}{\omega_1}.$$  A few of these polynomial relations we will use later.  When the lattice is understood, we will write $e_{2k}:=e_{2k}(\Lambda)$, $g_{a,b}:=g_{a,b}(\Lambda)$, and so on. \\
\begin{prop}
\label{Eisrelations}
\begin{align*}
&1. \ g_{3,0}=6e_{4}, \\
&2. \ g_{2,1}= -[e_2^*]^2+5e_4, \\
&3. \ g_{5,0}=120e_6, \\
&4. \ g_{4,1}= -5g_{3,0}g_{1,0}+\dfrac{7}{10}g_{5,0}, \\
&5. \ g_{3,2}=-2g_{2,1}g_{1,0}+\dfrac{5}{6}g_{4,1},\end{align*}
\end{prop}

\begin{proof} These follow immediately from \cite{Pol-E} Prop. 2.6.1. \end{proof}

We recall that for $k>1$, $e_{2k}$ are holomorphic and modular of weight $2k$, while $e_2^*$ is modular of weight 2. \\

When $C=\C/\langle\Z\oplus\tau\Z\rangle$ is a complex elliptic curve, we set $\Lambda=\Z+\tau\Z$ for the purpose of computing the Eisenstein series of the curve.  We set $t=\dfrac{\Im \tau}{\pi}$, and for non-negative integers $a,b,c,d$ define
$$M(a,b,c,d):=(-1)^{\binom{a+b+c+d+1}{2}}\dfrac{1}{a!b!c!d!}\cdot t^{a+b+c+d+1}\cdot g_{a+c,b+d}.$$ 

\subsection{Hochschild cohomology}
\label{Hcom}

We recall the appropriate definitions (see \cite{Manin}).  In particular if $k$ is a field, $B$ a $k$-algebra, and $R\subset B$ a semisimple ring such that $B=R\oplus B_+$ as an $R$-module, then we define here the reduced Hochschild cochain complex over $R$.  If $M$ is any $B$-bimodule, we define $$C^n(B,M)=\Hom_R(B_+^{\otimes n},M),$$ where $\otimes$ is tensor over $R$ unless otherwise specified.  Similarly $\Hom$ will mean $\Hom_R$ by default.  The differential $\delta:C^n\to C^{n+1}$ is defined by the equation,
\begin{align*}
\delta(\phi)(a_0,a_1,\ldots,a_n) &=a_0\phi(a_1,\ldots,a_n)+\sum_{i=1}^n(-1)^i\phi(a_0,\ldots,a_{i-1}a_i,\ldots,a_n)+ \\ &(-1)^{n+1}\phi(a_0,\ldots,a_{n-1})a_n.\end{align*} 
The cohomology of this complex is the Hochschild cohomology of $B$ with coefficients in $M$.  When $M=B$, we call the result simply the Hochschild cohomology of $B$. \\

For a graded $B$-bimodule $M$, the Hochschild complex $C^\bullet(B,M)$ is bigraded, $$\oplus_{n\geq 0}C^n(B,M)=\oplus_{n\geq 0,m\in\Z}C^n_{(m)}(B,M).$$  The Hochschild differential preserves $m$ and increases $n$ by 1.  Therefore we may consider for a fixed $m$ the complex $C^\bullet_{(m)}(B,M)$ and compute its cohomology. \\

\section{Hochschild cohomology and $A_\infty$-structures}
\label{HCAinf}

We recall here the relationship between the Hochschild cohomology of an associative algebra $A$ and the extensions of $A$ to $A$-infinity structures with $m_1=0$.  (See \cite{AAEK}, \cite{Pol-H}.) \\
\begin{lemma}
{lem3-k}
Let $A$ be a graded $k$-algebra, $\operatorname{char} k\neq 2,3$.  Let $m_1=0,m_2,\ldots,m_{n-1}$ be maps $m_k:A^{\otimes k}\to A$ homogeneous of internal degree $2-k$ and satisfying all corresponding $A_\infty$-relations.  Let $\delta$ be the differential on Hochschild cochains.  The $A_\infty$-relations can be rewritten in the form $$\delta m_k=\phi_k(m_3,\ldots,m_{k-1}),\ \ \ (*)$$
where 
\begin{enumerate}
\item $\phi_k$ is a quadratic expression;

\item $\phi_k(m_3,\ldots,m_{k-1}):A^{\otimes k+1}\to A$ is homogeneous of internal degree $2-k$; and 

\item $\delta\phi_k(m_3,\ldots,m_{k-1})=0$.  
\end{enumerate}
\end{lemma}

The lemma implies that the vanishing of $HH^{k+1}_{(2-k)}(A)$ guarantees that $m_2,\ldots,m_{k-1}$ can be extended with $m_k$ (i.e., the equation $(*)$ can be solved for $m_k$).  We prove here that $\phi_k(m_3,\ldots,m_{k-1})$ is a cocycle, as this was not located in the literature. \\

\begin{proof} Each $m_i$, $i=2,\ldots,k-1$, defines a map $m_i:T(A)\to A$ of internal degree $2-n$ equal to $m_i$ on $A^{\otimes i}$ and zero otherwise.  Following \cite{PenS}, let $V=(SA)^*$, The $m_i$ correspond to maps $d_i:V\to T(V)$ of internal degree 1.  Each $d_i$ can be uniquely extended to a superderivation $\hat{d_i}:T(V)\to T(V)$, also homoegeneous of internal degree 1.

The $A_\infty$-relations arise from certain relations among the brackets of the derivations $\hat{d_i}$.  The Lie bracket on derivations becomes the well-known Gerstenhaber bracket on Hochschild cochains in the dual picture.  The Hochschild differential $\delta$ acts by $\delta f=[m_2,f]$ and
$$\phi_k(m_3,\ldots,m_{k-1})=\left\{\begin{array}{cc} -([m_3,m_{k-1}]+[m_4,m_{k-2}]+\\
\cdots+[m_{(k+1)/2},m_{(k+3)/2}]) & \textrm{if $k$ is odd} \\ \\
-([m_3,m_{k-1}]+[m_4,m_{k-2}]+\\ \cdots+ \dfrac{1}{2}[m_{(k+2)/2},m_{(k+2)/2}]) & \textrm{if $k$ is even.}\end{array}\right.$$  The condition that $\phi_k(m_3,\ldots,m_{k-1})$ is a cocycle is equivalent to the condition that \\ $[d_2,\phi_k(m_3,\ldots,m_{k-1})^*]=0$.  ($\phi_k(m_3,\ldots,m_{k-1})^*$ is a derivation since its dual is defined by brackets of derivations.) \\

We proceed by induction in $k$.  For compactness we write $\phi_j^*:=\phi_j(m_3,\ldots,m_{j-1})^*$.  The base case $k=3$ is obvious since $\phi_3^*=0$.  Suppose $d_2,d_3,\ldots,d_{k-1}$ are defined such that $[d_2,d_j]=\phi_j^*$ for all $j<k$ with $k$ even.  Then
\begin{align*}
[d_2,\phi_k^*] &=-[d_2,[d_3,d_{k-1}]+[d_4,d_{k-2}]+\cdots+\dfrac{1}{2}[d_{(k+2)/2},d_{(k+2)/2}]] \\
&=[d_3,[d_2,d_{k-1}]+[d_{k-1},[d_2,d_3]]+[d_4,[d_2,d_{k-2}]]+[d_{k-2},[d_2,d_4]]+\\ &\cdots+[d_{(k+2)/2},[d_2,d_{(k+2)/2}] \\
&=\sum_{i=3}^{k-1}[d_i,\phi_{k+2-i}^*]
\end{align*}
When $\phi_{k+2-i}^*$ are expanded, the sum on the right will have terms:
\begin{enumerate}
\item $-[d_i,[d_j,d_t]]$, $-[d_j,[d_i,d_t]]$, $-[d_t,[d_i,d_j]]$ where $i+j+t=k+4$ and $i,j,t$ are all distinct.  Each term appears once (in $[d_i,\phi_{k+2-i}^*]$, $[d_j,\phi_{k+2-j}^*]$ and $[d_k,\phi_{k+2-t}^*]$ respectively) and their sum vanishes by the Jacobi identity. \\

\item $-[d_i,[d_i,d_j]]$, $-\dfrac{1}{2}[d_j,[d_i,d_i]]$ where $2i+j=k+4$ and $i\neq j$.  Each term appears once in $[d_i,\phi_{k+2-i}^*]$ and $[d_j,\phi_{k+2-j}^*]$, respectively.  By the Jacobi identity we have $$-\dfrac{1}{2}[d_j,[d_i,d_i]]=[d_i,[d_i,d_j]],$$ so these terms cancel. \\

\item $-[d_i,[d_i,d_i]]$ where $3i=k+4$.  By the Jacobi identity we have that $3[d_i,[d_i,d_i]]=0$, so $[d_i,[d_i,d_i]]=0$ so long as $\operatorname{char} k\neq 3$. \\
\end{enumerate}

When $k$ is odd we have
\begin{align*}
[d_2,\phi_k^*] &=-[d_2,[d_3,d_{k-1}]+\cdots+[d_{(k+1)/2},d_{(k+3)/2}]] \\
&=\sum_{i=3}^{k-1}[d_i,[d_2,d_{k+2-i}]] \\
&=\sum_{i=3}^{k-1}[d_i,\phi_{k+2-i}^*],
\end{align*}
with the same result as the $k$ even case.
\end{proof}

\begin{lemma} (\cite{Pol-H}, Lemma 2.2)  Let $m=(m_n)$ and $m^\prime=(m_n^\prime)$ be two admissible ($m_1=0$) $A_\infty$-structures on $A$ such that $m_i=m_i^\prime$ for $i<k$, where $k\geq 3$.  Then $m_k^\prime-m_k$ is a Hochschild cocycle. \\

Furthermore if $m_n^\prime-m_n$ is a coboundary, we can construct a strict $A_\infty$-isomorphism $f:A\to A$ such that $f*m_i^\prime=m_i^\prime$ for $i<n$ and $f*m_n^\prime=m_n$.\end{lemma}

Therefore if $HH^k_{(2-k)}(A)=0$, all extensions of $m_1=0,m_2,\ldots,m_{k-1}$ to $m_k$ are equivalent in this precise sense.

\begin{lemma} (\cite{Pol-H}, Lemma 2.3)  Let $m,m^\prime$ be two admissible $A_\infty$-structures on $A$, $f,f^\prime$ a pair of strict $A_\infty$-isomorphisms from $m$ to $m^\prime$ with $f_i=f_i^\prime$ for $i<k$, where $k\geq 2$.  Then $f_k^\prime-f_k$ is a Hochschild cocycle.  \\

Furthermore if $f_k^\prime-f_k$ is a coboundary then there is a homotopy $\phi$ such that $\phi*f_i=f_i$ for $i<k$ and $\phi*f_k=f_k^\prime$.\end{lemma}

 Therefore if $HH^k_{1-k}(A)=0$, all extensions of $f_1,\ldots,f_{k-1}$ (the start of a strict $A_\infty$-isomorphism from $m$ to $m^\prime$) to $f_k$ are homotopic in this precise sense. \\

\section{Definition of the graded associative algebra $B$}
\label{defalg}

Let $C$ be an elliptic curve over a field $k$, $\oc$ the structure sheaf of the curve, $P\in C$ a closed $k$-point, and $L=\oc(P)$ a line bundle of degree 1.  Let $B=\Ext^*(\oc\oplus L,\oc\oplus L)$.  $B$ is the direct sum of the components, \\
\begin{enumerate}
\item[(i)] $\Hom(\oc,\oc)$ and $\Hom(L,L)$, both one-dimensional generated by the identity maps $\id_\oc$, $\id_L$; \\
\item[(ii)] $\Hom(\oc,L)$, a one-dimensional space, generated by a function $\theta$; \\
\item[(iii)] $\Ext^1(L,\oc)$, a one-dimensional space generated by a function $\eta$; \\
\item[(iv)] $\Ext^1(L,L)$ and $\Ext^1(\oc,\oc)$, both isomorphic to the one-dimensional space $H^1(\oc)$. \\
\end{enumerate}
By Serre duality the products $\theta\eta=\xi\in\Ext^1(\oc,\oc)$ and $\eta\theta=\xi_L\in\Ext^1(L,L)$ are nonzero, so we take $\xi$ and $\xi_L$ as generators of those spaces.  For degree reasons all other products (except those involving the identities) are zero.  \\

Thus we consider the $k$-algebra $B=B_0\oplus B_1$, graded with $k$-basis $$B_0=\langle \id_{L},\id_{\oc},\theta\rangle, \ B_1=\langle \eta,\xi,\xi_L\rangle,$$ and nontrivial products 
$$(\id_{L})^2=\id_{L}, \ \  (\xi_L)(\id_{L}) =(\id_{L})(\xi_L)=\eta\theta=\xi_L, \ \ (\id_{L})\eta =\eta(\id_{\oc})=\eta, \ \
(\id_{\oc})^2 =\id_{\oc},$$  $$\xi(\id_{\oc}) =(\id_{\oc})\xi=\theta\eta=\xi, \ \ (\id_{\oc})\theta =\theta(\id_{L})=\theta.$$

We will also express $B$ as $B=R\oplus B_+$ where $R=\langle \id_{L},\id_{\oc}\rangle$ and $B_+=\langle \theta,\eta,\xi,\xi_L\rangle$, and in this way treat $B$ as an $R$-algebra. \\

\section{Computing Hochschild cohomology}
\label{Hochcomp}

We start by computing Hochschild cohomology of $B$ with coefficients in $(\xi_L,\xi)$, $B_1/(\xi_L,\xi)$, $B_+/B_1$, and $B/B_+$.  We use the notation 
\begin{align*}
C^\bullet(B,(\xi_L,\xi))&=C^\bullet(\xi_L,\xi), \\
C^\bullet(B,B_1/(\xi_L,\xi))&=C^\bullet(\eta), \\
C^\bullet(B,B_+/B_1)&=C^\bullet(\theta), \\
C^\bullet(B,B/B_+)&=C^\bullet(\id_{L},\id_{\oc}).
\end{align*} \\

After making these calculations we will use the long exact sequences on cohomology associated to the short exact sequences (of $R$-bimodules)
\begin{align*}
&0\to (\xi_L,\xi)\to B_1\to B_1/(\xi_L,\xi)\to 0, \\
&0\to B_+/B_1\to B/B_1\to B/B_+\to 0, \\
&0\to B_1\to B\to B/B_1\to 0,
\end{align*}
to finally compute $HH^n_{(1-n)}(B), HH^n_{(2-n)}(B),$ and $HH^n_{(3-n)}(B)$.

\subsection{$HH^\bullet(\xi_L,\xi)$ and $HH^\bullet(\id_L,\id_{\oc}$)}
\label{ids}

Since $$C^n(\xi_L,\xi)=C^n(\xi_L)\oplus C^n(\xi):=C^n(B,(\xi_L))\oplus C^n(B,(\xi)),$$ we can consider each complex separately.  Because the first and last terms of the Hochschild differential are zero, the cochain complex $(C^\bullet(\xi_L),\delta)$ is dual to the chain complex $(C_\bullet(L),-d)$ where $$C_n(L):=\id_{L}\otimes_R(B_+^{\otimes n})\otimes_R\id_{L}, \ \ d(a_1\otimes\cdots\otimes a_n)=\sum_{i=1}^{n-1}(-1)^i a_1\otimes\cdots a_ia_{i+1}\otimes\cdots a_n.$$  The chain complex $C_\bullet(L)$ has an internal grading given by $\deg(a_1\otimes\cdots\otimes a_n)=\sum \deg a_i$.  We let $C_n^{(m)}(L)$ be the space of tensors of internal degree $m$ in $C_n(L)$.  The differential preserves $m$ and decreases $n$ by 1. \\

\begin{prop}\label{iddelpprop} \begin{align*}
H_n(C_\bullet^{(n)}(L)) &=0 \ \textrm{for all $n$.}\\ 
H_n(C_\bullet^{(n-1)}(L)) &=\left\{\begin{array}{cc}k & \textrm{if $n=3,4$} \\ 0 & \textrm{otherwise.}\end{array}\right.\\ 
H_n(C_\bullet^{(n-2)}(L)) &=\left\{\begin{array}{cc} k & \textrm{if $n=7,8$} \\ 0 & \textrm{otherwise.}\end{array}\right.\\
H_n(C_\bullet^{(n-3)}(L)) &=\left\{\begin{array}{cc}k & \textrm{if $n=11,12$} \\ 0 & \textrm{otherwise.}\end{array}\right.\end{align*} \end{prop}

For $m\geq 0$.  We consider a decreasing filtration on the complex $C_{\bullet}^{(m)}(L)$, letting $$F^iC_{n}^{(m)}(L)=\langle (\xi_L)^{k_1}\eta(\xi)^{c_1}\theta(\xi_L)^{k_2}\cdots \theta(\xi_L)^{k_{n-m}}|\sum k_j\geq i\rangle.$$ \\

For fixed $n,m$, the space $F^iC_{n}^{(m)}(L)=0$ for $i\gg0$.   The spectral sequence of this filtration therefore converges to the homology of the complex.  The zero page has for each $i\geq 0$ a complex $$\gr_iC_\bullet^{(m)}(L)=F^iC_n^{(m)}(L)/F^{i-1}C_n^{(m)}(L)= \langle (\xi_L)^{k_1}\eta(\xi)^{c_1}\theta(\xi_L)^{k_2}\cdots \theta(\xi_L)^{k_{n-m}}|
\sum k_j=i\rangle.$$ \\

The proof of Proposition~\ref{iddelpprop} will follow from Lemmas~\ref{delplem1} through ~\ref{delplem5}.

\begin{lemma}\label{delplem1} $$H_{n-1}(\gr_iC_\bullet^{(n-1)}(L))=\left\{\begin{array}{cc} k & \textrm{if $i=n-1$} \\ 0 & \textrm{otherwise.}\end{array}\right.$$  For $i=n-1$, $(\xi_L)^{n-1}$ is a generating cycle. \end{lemma}

\begin{proof} Only for $i=n-1$ is $\gr_iC_{n-1}^{(n-1)}(L)\neq 0$.  Then $$\gr_{n-1}C_{n-1}^{(n-1)}(L)=\langle(\xi_L)^{n-1}\rangle,$$ and clearly $d(\xi^{n-1})=0$.  Since $\gr_{n-1}C_n^{(n-1)}(L)=0$, $(\xi_L)^{n-1}$ is not in $\im d$. \end{proof}

\begin{lemma}\label{delplem2} $H_n(\gr_iC_\bullet^{(n-1)}(L))=0$ for $i\neq n-2$ and $$H_n(\gr_{n-2}C_\bullet^{(n-1)}(L))=\langle(\xi_L)^{a}\eta\theta(\xi_L)^{c}|a+c=n-2\rangle$$ \end{lemma}

\begin{proof} For $i>n-2$, $\gr_iC_n^{(n-1)}(L)=0$.  For $i<n-2$, we have $$\gr_iC_n^{(n-1)}(L)=\langle (\xi_L)^{a}\eta (\xi)^b\theta(\xi_L)^{c}|a+c=i, b=n-2-i\rangle.$$ Since $b\neq 0$ this entire space is in $\ker d$ (no nontrivial adjacent products).  The whole space is also in $\im d$ since $(\xi_L)^{a}\eta\theta\eta (\xi)^{b-1}\theta(\xi_L)^{c}\in \gr_iC_{n+1}^{(n-1)}(L)$ and $$d((\xi_L)^{a}\eta\theta\eta (\xi)^{b-1}\theta(\xi_L)^{c})=\pm(\xi_L)^{a+1}\eta (\xi)^{b-1}\theta(\xi_L)^{c}\mp(\xi_L)^a\eta (\xi)^b\theta(\xi_L)^c,$$ where the first term on the right is 0 in the quotient.  This proves the first claim.  \\

(Note for future reference: this same trick will show that any tensor of the form $$(\xi_L)^{a_1}\eta(\xi)^{b_1}\theta\cdots(\xi)^{b_{n-1}}\theta(\xi_L)^{a_n}$$ such that all of $a_2,\ldots,a_{n-1}\neq 0$ and at least one $b_i\neq 0$ is both in $\ker d$ and in $\im d$.  We will subsequently refer to this fact when needed the Lemma~\ref{delplem2} trick.) \\

We have $$\gr_{n-2}C_n^{(n-1)}(L)=\langle (\xi_L)^{a}\eta\theta(\xi_L)^c|a+c=n-2\rangle.$$  This whole space is in $\ker d$ since $d$ adds one to the sum of powers on $\xi_L$.  Since $\gr_{n-2}C_{n+1}^{(n-1)}(L)=0$ for degree reasons, nothing in this space is in $\im d$. \end{proof}

\begin{lemma}\label{delplem3} $H_{n+1}(\gr_iC_\bullet^{(n-1)}(L))=0$ for $i\neq n-3$ and $$H_{n+1}(\gr_{n-3}C_\bullet^{(n-1)}(L))=\langle(\xi_L)^a\eta\theta(\xi_L)^c\eta\theta(\xi^L)^e|a+c+e=n-3, c\neq 0\rangle$$ \end{lemma}

\begin{proof} For $i>n-3$, $\gr_iC_{n+1}^{(n-1)}(L)=0$. \\

For $i=n-3$, it is clear that $(\xi_L)^a\eta\theta(\xi_L)^c\eta\theta(\xi^L)^e$ is a class in homology when $c\neq 0$.  When $c=0$ we have $$d((\xi_L)^{a}\eta\theta\eta\theta(\xi_L)^{e})=\pm(\xi_L)^{a}\eta(\xi)\theta(\xi_L)^{e}$$ with the other terms vanishing in the quotient.  Different choices of $a,e$ therefore produce linearly independent boundaries, so no linear combination of such elements may be in $\ker d$. \\

For $i<n-3$ there are two cases. \\

\begin{itemize}
\item[(a)] $c\neq 0$.  Then $$d((\xi_L)^a\eta(\xi)^b\theta(\xi_L)^c\eta(\xi)^d\theta(\xi_L)^e)=0.$$  Since either $b\neq 0$ or $d\neq 0$, this tensor is in $\im d$ by the Lemma~\ref{delplem2} trick. \\

\item[(b)]\label{simplex} $c=0$.  Then $$d((\xi_L)^a\eta(\xi)^b\theta\eta(\xi)^d\theta(\xi_L)^{e})=\pm(\xi_L)^a\eta(\xi)^{b+d+1}\theta(\xi_L)^e.$$  Thus for fixed $a,e$ and different choices of $b,d$ with $b+d=n-3-i$, we get an element of $\ker d$ by taking some combination of these elements.  Tensors of the form $(\xi_L)^a\eta(\xi)^b\theta\eta(\xi)^d\theta(\xi_L)^{e}$ appear in boundaries of tensors $(\xi_L)^a\eta(\xi)^{c_1}\theta\eta(\xi)^{c_2}\theta\eta(\xi)^{c_3}\theta(\xi_L)^e),$ for some choice of $c_1,c_2,c_3$ where $\sum c_i=n-4-i$.  More generally, the space of tensors with all interior powers of $\xi_L$ equal to 0, as here, form their own complex.  Here $d$ has the form
$$d(c_1,\ldots,c_m)=\sum_{i=1}^{m-1}\pm(c_1,\ldots,c_i+c_{i+1}+1,\ldots,c_m),$$
where $\sum_{i=1}^mc_i=n-1-i-m.$
Making a change of variables $c_i^\prime=c_i+1$ we have
$$d(c_1^\prime,\ldots,c_m^\prime)= \sum_{i=1}^{m-1}\pm(c_1^\prime,\ldots,c_i^\prime+c_{i+1}^\prime,\ldots,c_m^\prime),$$
where $\sum_{i=1}^mc_i^\prime=n-1-i$.
Finally letting $\overline{c_i}=\sum_{j=1}^ic_i^\prime$ and mapping $(c_1^\prime,\ldots,c_m^\prime)$ to $\{\overline{c_1},\ldots,\overline{c_{m-1}}\}$ for all $m$ gives an isomorphism with the reduced complex of the $(n-2-i)$-simplex.  In this calculation we seek $\tilde{H}_0$ of the simplex, which is of course 0.
\end{itemize}

(Note for future reference: In subsequent sections we will encounter complexes which can be mapped to simplicial complexes in essentially this same way.  When this is the case, we will refer to this procedure as the simplex trick.)

\end{proof}

\begin{lemma}\label{delplem4} $H_{n+2}(\gr_iC_\bullet^{(n-1)}(L))=0$ for $i\neq n-4$ and $$H_{n+2}(\gr_{n-4}C_\bullet^{(n-1)}(L))=\langle(\xi_L)^a\eta\theta(\xi_L)^c\eta\theta(\xi_L)^e\eta\theta(\xi_L)^g|a+c+e+g=n-4 \ \textrm{and} \ c,e\neq 0\rangle.$$ \end{lemma}

\begin{proof} For $i>n-4$, $\gr_iC_\bullet^{(n-1)}(L)=0$.  For $i<n-4$ all tensors must have at least one $\xi$.  Such tensors already in $\ker d$ are in $\im d$ by the Lemma~\ref{delplem2} trick.  Thus additional cases occur only when we have some linear combinations that produce elements of $\ker d$.  To examine this we calculate,
\begin{align*}
d((\xi_L)^a\eta(\xi)^b\theta\eta(\xi)^d\theta(\xi_L)^e\eta(\xi)^f\theta(\xi_L)^g) & =\pm(\xi_L)^a\eta(\xi)^{b+d+1}\theta(\xi_L)^e\eta(\xi)^f\theta(\xi_L)^g, \ \ &(1)\\
d((\xi_L)^a\eta(\xi)^b\theta(\xi_L)^c\eta(\xi)^d\theta\eta(\xi)^f\theta(\xi_L)^g) & =\pm(\xi_L)^a\eta(\xi)^b\theta(\xi_L)^c\eta(\xi)^{d+f+1}\theta(\xi_L)^g, \ \ &(2)\\
d((\xi_L)^a\eta(\xi)^b\theta\eta(\xi)^d\theta\eta(\xi)^f\theta(\xi_L)^g) & =\pm(\xi_L)^a\eta(\xi)^{b+d+1}\theta\eta(\xi)^f\theta(\xi_L)^g\pm \\ &(\xi_L)^a\eta(\xi)^b\theta\eta(\xi)^{d+f+1}\theta(\xi_L)^g \ \ &(3)
\end{align*}

Some combination of two tensors from (1) for fixed $a,e,f$ and different choices of $b,d$ are be in $\ker d$.  Similarly for pairs of tensors in (2).  Both cases are covered by the simplex argument of Lemma~\ref{delplem3} by tensoring the complex from that lemma with $\eta(\xi)^f\theta(\xi_L)^g$ on the right in the former case, or with $(\xi_L)^a\eta(\xi)^b\theta$ on the left in the latter.  \\

Finally we must consider combinations of a tensor from (1) with powers $b,d,f$ and a tensor from (2) with powers $b^\prime,d^\prime,f^\prime$ such that $c=e$, $b+d+1=b^\prime$ and $f=b^\prime+d^\prime+1$.  In that case,
\begin{align*}
&d((\xi_L)^a\eta(\xi)^b\theta\eta(\xi)^d\theta(\xi_L)^e\eta(\xi)^{d^\prime}\theta\eta(\xi)^{f-d^\prime-1}\theta(\xi_L)^g)= \\
&\pm(\xi_L)^a\eta(\xi)^{b+d+1}\theta(\xi_L)^e\eta(\xi)^{d^\prime}\theta\eta(\xi)^{f-d^\prime-1}\theta(\xi_L)^g\pm(\xi_L)^a\eta(\xi)^b\theta\eta(\xi)^d\theta(\xi_L)^e\eta(\xi)^f\theta(\xi_L)^g
\end{align*}
where $f-d^\prime-1=f^\prime$. \\

Tensors of type (3) form a part of the complex considered in the simplex argument from~\ref{delplem3}.  Here we must calculate $H_1$ of the simplex, which is 0. \\

To prove the claim for $i=n-4$, note that all of the tensors described are obviously in $\ker d$.  Other possible tensors and their boundaries are,
\begin{align*}
d((\xi_L)^a\eta\theta\eta\theta(\xi_L)^e\eta\theta(\xi_L)^g) & =\pm(\xi_L)^a\eta(\xi)\theta(\xi_L)^e\eta\theta(\xi_L)^g, \\
d((\xi_L)^a\eta\theta(\xi_L)^c\eta\theta\eta\theta(\xi_L)^g) & =  \pm(\xi_L)^a\eta\theta(\xi_L)^c\eta(\xi)\theta(\xi_L)^g, \\
d((\xi_L)^a\eta\theta\eta\theta\eta\theta(\xi_L)^g) & = \pm(\xi_L)^a\eta(\xi)\theta\eta\theta(\xi_L)^g\pm(\xi_L)^a\eta\theta\eta(\xi)\theta(\xi_L)^g
\end{align*}

For different choices of $a,c,e,g$ such that the appropriate sums equal $n-4$, the set of tensors on the right are linearly independent and so produce no class in homology. \end{proof}

\begin{lemma}\label{delplem5}\begin{align*}H_{n+3}(\gr_{n-5}C_\bullet^{(n-1)}(L))&=\langle(\xi_L)^{k_1}\eta\theta(\xi_L)^{k_2}\eta\theta(\xi_L)^{k_3}\eta\theta(\xi_L)^{k_4}\eta\theta(\xi_L)^{k_5}|\sum k_i=n-5 \\ &\ \textrm{and} \ k_2,k_3,k_4\neq 0\rangle.\end{align*}\end{lemma}

\begin{proof} These elements are all clearly in $\ker d$ and cannot be in $\im d$.  The boundaries of tensors with either $k_2=0,k_3=0,k_4=0$ are all linearly independent, similar to the previous lemma. \end{proof}

\begin{proof}[Proof of Proposition~\ref{iddelpprop}] On the first page of the spectral sequence we have from this collection of lemmas one nontrivial complex, 
\begin{align*}
0&\to H_{n+3}(\gr_{n-5}C_\bullet^{(n-1)}(L))\to H_{n+2}(\gr_{n-4}C_\bullet^{(n-1)}(L))\to H_{n+1}(\gr_{n-3}C_\bullet^{(n-1)}(L)) \\&\to H_n(\gr_{n-2}C_\bullet^{(n-1)}(L))\to H_{n-1}(\gr_{n-1}C_\bullet^{(n-1)}(L))\to 0.
\end{align*}

The differential takes the form $$(k_1,\ldots,k_m)\mapsto (k_1+k_2+1,k_3,\ldots)\pm\cdots\pm(k_1,\ldots,k_{m-1}+k_m+1).$$  We use the simplex trick as in Lemma~\ref{delplem3}, to map to a simplicial complex.  In particular, we map isomorphically to the dimension $3,2,1,0,-1$ part of the simplicial complex $\Delta[n-1]$ considered in Appendix~\ref{simpsec}.  The condition that the middle $k_i$ are nonzero is equivalent to the middle $k_i^\prime\geq 2$, thus the difference of at least two between adject vertices.  Since $\sum_{i=1}^mk_i=n-m$ and $k_m=0$ is possible, it follows that $\sum_{i=1}^{m-1}k_i^\prime=n-1$.  \\

From Proposition~\ref{simpprop}, the resulting simplicial complex has (reduced) homology in dimension 0 for $n-1=2,3$ (so the chain complex in dimension $n$ for $n=3,4$); in dimension 1 for $n-1=5,6$ (so the chain complex in dimension $n+1$ for $n+1=7,8$); and in dimension 2 for $n-1=8,9$ (so the chain complex in dimension $n+2$ for $n+2=11,12$).  This is the result.  \end{proof}

Furthermore the correspondence with the simplicial complex allows us to find explicit representatives of all classes.
$$\begin{array}{|c|l|} \hline
\textrm{location} & \textrm{representative} \\ \hline
H_3^{(2)}(L) & \sigma_3^{(2)}:=\eta\theta(\xi_L)+(\xi_L)\eta\theta \\ \hline
H_4^{(3)}(L) & \sigma_4^{(3)}:=\eta\theta(\xi_L)^2+(\xi_L)\eta\theta(\xi_L) \\ \hline
H_7^{(5)}(L) & \sigma_7^{(5)}:=\eta\theta(\xi_L)^2\eta\theta\xi_L+\eta\theta(\xi_L)^3\eta\theta+(\xi_L)\eta\theta(\xi_L)\eta\theta\xi_L+(\xi_L)\eta\theta(\xi_L)^2\eta\theta \\ \hline
H_8^{(6)}(L) & \sigma_8^{(6)}:=\sigma_7^{(5)}\otimes\xi_L \\ \hline
H_{11}^{(8)}(L) & \sigma_{11}^{(8)}:=\eta\theta(\xi_L)^3\eta\theta(\xi_L)\eta\theta\xi_L+\eta\theta(\xi_L)^2\eta\theta(\xi_L)^2\eta\theta(\xi_L)+(\xi_L)\eta\theta(\xi_L)^2\eta\theta(\xi_L)\eta\theta\xi_L +\\
 & (\xi_L)\eta\theta(\xi_L)\eta\theta(\xi_L)^2\eta\theta\xi_L+\eta\theta(\xi_L)^3\eta\theta(\xi_L)^2\eta\theta+\eta\theta(\xi_L)^2\eta\theta(\xi_L)^3\eta\theta+ \\
& (\xi_L)\eta\theta(\xi_L)^2\eta\theta(\xi_L)^2\eta\theta+(\xi_L)\eta\theta(\xi_L)\eta\theta(\xi_L)\eta\theta \\ \hline 
H_{12}^{(9)}(L) & \sigma_{12}^{(9)}:=\sigma_{11}^{(8)}\otimes\xi_L \\ \hline
\end{array}$$

We would like to find representatives that are a little more manageable, and they are available. \\
\begin{align*}
\sigma_3^{(2)}-d(\eta\theta\eta\theta) & =\eta(\xi)\theta,\\
\sigma_4^{(3)}-d(\eta\theta\eta\theta\xi_L) & =\eta(\xi)\theta\xi_L,\\
\sigma_7^{(5)}-d(\eta\theta\eta\theta(\xi_L)\eta\theta\xi_L)-d(\eta\theta\eta\theta(\xi_L)^2\eta\theta)-d(\eta(\xi)\theta(\xi_L)\eta\theta\eta\theta) & =\eta(\xi)\theta(\xi_L)\eta(\xi)\theta, \\
\sigma_8^{(6)} & \sim \eta(\xi)\theta(\xi_L)\eta(\xi)\theta\xi_L, \\
\sigma_{11}^{(8)}-d(\eta\theta\eta\theta(\xi_L)^2\eta\theta(\xi_L)\eta\theta\xi_L)-d(\eta(\xi)\theta(\xi_L)\eta\theta\eta\theta(\xi_L)\eta\theta\xi_L)-& = \eta(\xi)\theta(\xi_L)\eta(\xi)\theta(\xi_L)\eta(\xi)\theta, \\
d(\eta(\xi)\theta(\xi_L)\eta\theta\eta\theta(xi_L)^2\eta\theta)-d(\eta(\xi)\theta(\xi_L)\eta(\xi)\theta(\xi_L)\eta\theta\eta\theta) &   \\
\sigma_{12}^{(9)} & \sim gef\xi gef\xi gef\xi.
\end{align*}
So an updated version of representatives is
$$\begin{array}{|c|l|} \hline
H_3^{(2)}(L) & \eta(\xi)\theta \\ \hline
H_4^{(3)}(L) & \eta(\xi)\theta\xi_L\sim(\xi_L)\eta(\xi)\theta \\ \hline
H_7^{(5)}(L) & \eta(\xi)\theta(\xi_L)\eta(\xi)\theta \\ \hline
H_8^{(6)}(L) & \eta(\xi)\theta(\xi_L)\eta(\xi)\theta\xi_L\sim(\xi_L)\eta(\xi)\theta(\xi_L)\eta(\xi)\theta \\ \hline
H_{11}^{(8)}(L) & \eta(\xi)\theta(\xi_L)\eta(\xi)\theta(\xi_L)\eta(\xi)\theta \\ \hline
H_{12}^{(9)}(L) & \eta(\xi)\theta(\xi_L)\eta(\xi)\theta(\xi_L)\eta(\xi)\theta\xi_L\sim(\xi_L)\eta(\xi)\theta(\xi_L)\eta(\xi)\theta(\xi_L)\eta(\xi)\theta\\ \hline
\end{array}$$

\begin{corollary}\label{idocprop}
\begin{align*}
H_n(C_\bullet^{(n)}(\oc)) &=0 \ \textrm{for all $n$.}\\
H_n(C_\bullet^{(n-1)}(\oc)) &=\left\{\begin{array}{cc}k & \textrm{if $n=3,4$} \\ 0 & \textrm{otherwise.}\end{array}\right.\\
H_n(C_\bullet^{(n-2)}(\oc)) &=\left\{\begin{array}{cc} k & \textrm{if $n=7,8$} \\ 0 & \textrm{otherwise.}\end{array}\right.\\
H_n(C_\bullet^{(n-3)}(\oc)) &=\left\{\begin{array}{cc}k & \textrm{if $n=11,12$} \\ 0 & \textrm{otherwise.}\end{array}\right.\\
\end{align*}
\end{corollary}

\begin{proof} There is an isomorphism of complexes $C_\bullet^{(n-1)}(L)\to C_\bullet^{(n-1)}(\oc)$ via $$\eta\mapsto \theta, \ \theta\mapsto \eta, \ \xi\mapsto \xi_L, \ \xi_L\mapsto \xi.$$    \end{proof}

So we get representatives,
$$\begin{array}{|c|l|} \hline
H_3^{(2)}(\oc) & \theta(\xi_L)\eta \\ \hline
H_4^{(3)}(\oc) & \theta(\xi_L)\eta(\xi)\sim\xi\theta(\xi_L)\eta \\ \hline
H_7^{(5)}(\oc) & \theta(\xi_L)\eta(\xi)\theta(\xi_L)\eta \\ \hline
H_8^{(6)}(\oc) & \theta(\xi_L)\eta(\xi)\theta(\xi_L)\eta(\xi)\sim\xi\theta(\xi_L)\eta(\xi)\theta(\xi_L)\eta \\ \hline
H_{11}^{(8)}(\oc) & \theta(\xi_L)\eta(\xi)\theta(\xi_L)\eta(\xi)\theta(\xi_L)\eta \\ \hline
H_{12}^{(9)}(\oc) & \theta(\xi_L)\eta(\xi)\theta(\xi_L)\eta(\xi)\theta(\xi_L)\eta(\xi)\sim\xi\theta(\xi_L)\eta(\xi)\theta(\xi_L)\eta(\xi)\theta(\xi_L)\eta \\ \hline
\end{array}$$

\begin{corollary}\label{HHids} 
\begin{align*}
HH^n_{(1-n)}(\xi_L,\xi)&=0 \ \textrm{for all $n$,}\\
HH^{n}_{(2-n)}(\xi_L,\xi)&=HH^n_{(1-n)}(\id_{L},\id_{\oc})=\left\{\begin{array}{cc} k^2 & \textrm{if $n=3,4$} \\ 0 & \textrm{otherwise} \end{array}\right.\\ 
HH^n_{(3-n)}(\xi_L,\xi)&=HH^n_{(2-n)}(\id_{L},\id_{\oc})=\left\{\begin{array}{cc}k^2 & \textrm{if $n=7,8$} \\ 0 & \textrm{otherwise}\end{array}\right.\\
HH^n_{(4-n)}(\xi_L,\xi)&=HH^n_{(3-n)}(\id_{L},\id_{\oc})=\left\{\begin{array}{cc}k^2 & \textrm{if $n=11,12$} \\ 0 & \textrm{otherwise}\end{array}\right.\end{align*} \end{corollary}

\begin{proof} The Hochschild complexes $C^\bullet_{(1-m)}(\xi_L)$ and $C^\bullet_{(-m)}(\id_L)$ are both dual to $C_\bullet^{(m)}(L)$ for all $m$.  The complexes $C^\bullet_{(1-m)}(\xi)$ and $C^\bullet_{(-m)}(\id_{\oc})$ are both dual to $C_\bullet^{(m)}(\oc)$.  \end{proof}

\subsection{$HH^\bullet(\eta)$ and $HH^\bullet(\theta)$.}
\label{gf}

Let $C_n^{(m)}(\eta)$ be the space spanned by tensors of degree $m$ in $\id_{L}\otimes_R B_+^{\otimes n}\otimes_R\id_{\oc},$ made into a complex with the differential $d$ from the previous section. \\

\begin{prop}\label{gprop} \begin{align*}
H_n(C_\bullet^{(n)}(\eta)) &=\left\{\begin{array}{cc} k & \textrm{if $n=1,2$} \\ 0 & \textrm{otherwise,}\end{array}\right.\\
H_n(C_\bullet^{(n-1)}(\eta)) &=\left\{\begin{array}{cc}k & \textrm{if $n=5,6$} \\ 0 & \textrm{otherwise,}\end{array}\right.\\
H_n(C_\bullet^{(n-2)}(\eta)) &=\left\{\begin{array}{cc}k & \textrm{if $n=9,10$} \\ 0 & \textrm{otherwise.}\end{array}\right.\\
H_n(C_\bullet^{(n-3)}(\eta)) &=\left\{\begin{array}{cc}k & \textrm{if $n=13,14$} \\ 0 & \textrm{otherwise.}\end{array}\right.\end{align*} \end{prop}

We consider the decreasing filtration with $$F^iC_n^{(m)}=\langle (\xi_L)^k_1\eta(\xi)^c_1\theta\cdots(\xi_L)^k_{n-m+1}\eta(\xi)^{k_{n-m+1}}|\sum k_j\geq i\rangle.$$  Then the zero page of the spectral sequence of the filtration has complexes such that the sum of powers of $\xi_L$ is exactly $i$.  The proof of Proposition~\ref{gprop} will follow from the following lemmas.

\begin{lemma}\label{glem1} $H_{n-1}(\gr_iC_\bullet^{(n-1)}(\eta))=\left\{\begin{array}{cc} k & \textrm{if $i=n-2$} \\ 0 & \textrm{otherwise}\end{array}\right.$\end{lemma} 

\begin{proof} The space $\gr_iC_{n-1}^{(n-1)}(\eta)$ is one-dimensional for $i\leq n-2$ spanned by $(\xi_L)^i\eta(\xi)^{n-2-i}$.  If $i<n-2$ then the power on $\xi$ is nonzero and thus this tensor is a boundary, and otherwise it is not. \end{proof}

\begin{lemma}\label{glem2} $H_n(\gr_iC_\bullet^{(n-1)}(\eta))=0$ for $i\neq n-3$ and $$H_n(\gr_{n-3}C_\bullet^{(n-1)}(\eta))=\langle(\xi_L)^{a}\eta\theta(\xi_L)^{c}\eta|a+c=n-3, \ c\neq 0\rangle.$$ \end{lemma}

\begin{proof} Consider $i<n-3$ first.  In this case there is at least one $\xi$ present and there are two cases: $c\neq 0, c=0$.  If $c\neq 0$ then we get an element of $\ker d$ and $\im d$ as usual. \\

If $c=0$ then $$d((\xi_L)^{a}\eta(\xi)^b\theta\eta(\xi)^d)=\pm(\xi_L)^a\eta(\xi)^{b+d+1}.$$ For fixed $a$ the complex of such tensors is isomorphic to that from Lemma~\ref{delplem3}, via tensoring this present complex with $\theta(\xi_L)^e$ on the right.  Thus this complex has no homology.\\

For $i=n-3$ the complex is generated by tensors $(\xi_L)^a\eta\theta(\xi_L)^c\eta$.   For $c\neq 0$ we get an element of $\ker d$ and not $\im d$ as usual.  For $c=0$ note that $$d((\xi_L)^a\eta\theta\eta)=\pm(\xi_L)^a\eta(\xi)\neq 0.$$
\end{proof}

\begin{lemma}\label{glem3} $H_{n+1}(\gr_iC_\bullet^{(n-1)}(\eta))=0$ for $i\neq n-4$ and $$H_{n+1}(\gr_{n-4}C_\bullet^{(n-1)}(\eta))=\langle(\xi_L)^a\eta\theta(\xi_L)^c\eta\theta(\xi_L)^e\eta|a+c+e=n-4, \ c,e\neq 0\rangle.$$ \end{lemma}

\begin{proof} 
For $i<n-4$ there are several cases to consider.  If both $c,e\neq 0$, we get an element of $\ker d$ and $\im d$.  The subcomplex where $c=0$ (or $e=0$) is isomorphic to the complex from Lemma~\ref{delplem4} by tensoring the present complex with $\theta(\xi_L)^g$ on the right.  Thus those subcomplexes have no homology. \\

Now consider $i=n-4$.  If at least one of $c,e$ is zero we are in one of the cases,
\begin{align*}
d((\xi_L)^a\eta\theta\eta\theta(\xi_L)^e\eta) & = \pm(\xi_L)^a\eta(\xi)\theta(\xi_L)^e\eta, \\
d((\xi_L)^a\eta\theta(\xi_L)^c\eta\theta\eta) & = \pm(\xi_L)^a\eta\theta(\xi_L)^c\eta(\xi), \\
d((\xi_L)^a\eta\theta\eta\theta\eta) & =  \pm(\xi_L)^a\eta(\xi)\theta\eta\pm(\xi_L)^a\eta\theta\eta(\xi).
\end{align*}
The expressions on the right are linearly independent for different choices of $a,c,e$.  If both $c,e\neq 0$ then we get an element of $\ker d$ and not $\im d$ as usual. \\
\end{proof}

\begin{lemma}\label{glem4} $H_{n+2}(\gr_iC_\bullet^{(n-1)}(\eta))=0$ for $i\neq n-5$, and $$H_{n+2}(\gr_{n-5}C_\bullet^{(n-1)}(\eta))=\left\langle\begin{array}{c}(\xi_L)^a\eta\theta(\xi_L)^c\eta\theta(\xi_L)^e\eta\theta(\xi_L)^g\eta| \\ a+c+e+g=n-5 \ \textrm{and} \ c,e,g\neq 0\end{array}\right\rangle.$$ \end{lemma}

\begin{proof} Indecomposable tensors in $\ker d$ (that is, where $c,e,g$ are all nonzero) are in $\im d$ for $i<n-5$ since they have some nonzero power of $\xi$. \\

With $i<n-5$ and at least one of $c,g,e=0$ we have,
\begin{align*}
d((\xi_L)^a\eta(\xi)^b\theta\eta(\xi)^d\theta(\xi_L)^e\eta(\xi)^f\theta(\xi_L)^g\eta(\xi)^h & =  \pm(\xi_L)^a\eta(\xi)^{b+d+1}\theta(\xi_L)^e\eta(\xi)^f\theta(\xi_L)^g\eta(\xi)^h, \ \ &(1)\\
d((\xi_L)^a\eta(\xi)^b\theta(\xi_L)^c\eta(\xi)^d\theta\eta(\xi)^f\theta(\xi_L)^g\eta(\xi)^h & =  \pm(\xi_L)^a\eta(\xi)^b\theta(\xi_L)^c\eta(\xi)^{d+f+1}\theta(\xi_L)^g\eta(\xi)^h, \ \ &(2)\\
d((\xi_L)^a\eta(\xi)^b\theta(\xi_L)^c\eta(\xi)^d\theta(\xi_L)^e\eta(\xi)^f\theta\eta(\xi)^h & =  \pm(\xi_L)^a\eta(\xi)^b\theta(\xi_L)^c\eta(\xi)^d\theta(\xi_L)^e\eta(\xi)^{f+h+1}. \ \ &(3)
\end{align*}
Combinations of pairs of tensors with fixed $a,e,g,f,h$ in line (1) are in $\im d$, following from the calculation in Lemma~\ref{delplem4}, by tensoring those with $\eta(\xi)^h$ on the right, and similarly for pairs from (2) and pairs from (3).  \\

Combinations of cross-terms from (1) and (2) in $\ker d$ are in $\im d$ following also from tensoring calculations from Lemma~\ref{delplem4} on the right with $\eta(\xi)^h$; similarly cross-terms from (2) and (3), dividing the calculation in Lemma~\ref{delplem4} by $\theta(\xi_L)^g$ (in an appropriate sense) on the right and tensoring on the left with $(\xi_L)^{a_0}\eta(\xi)^{b_0}\theta$.  It follows that all combinations of cross-terms from (1) and (3) whose boundaries vanish also appear in $\im d$. \\

All tensors with $i=n-5$ and $c,e,g\neq 0$ are in $\ker d$ and not $\im d$.  Then note that if any of $c,e,g=0$, the resulting boundaries are all linearly independent.  This fact is immediate after writing them down.
\end{proof}

\begin{lemma}\label{glem5} $$H_{n+3}(\gr_{n-6}C_\bullet^{(n-1)}(\eta))=\left\langle\begin{array}{c}(\xi_L)^a\eta\theta(\xi_L)^c\eta\theta(\xi_L)^e\eta\theta(\xi_L)^g\eta\theta(\xi_L)^i\eta| \\ a+c+e+g+i=n-6 \ \textrm{and} \ c,e,g,i\neq 0\end{array}\right\rangle.$$ \end{lemma}

\begin{proof} All of the described terms are obviously in $\ker d$ and not $\im d$.  Once again the boundaries of all terms with any of $c,e,g,i=0$ are linearly independent. \end{proof}

\begin{proof}[Proof of Proposition ~\ref{gprop}] From the lemmas we find that on the first page of the spectral sequence we have one non-trivial complex,
\begin{align*}H_{n+3}(\gr_{n-6}C_\bullet^{(n-1)}(\eta))\to H_{n+2}(\gr_{n-5}C_\bullet^{(n-1)}(\eta))\to H_{n+1}(\gr_{n-4}C_\bullet^{(n-1)}(\eta))\\ \to H_n(\gr_{n-3}C_\bullet^{(n-1)}(\eta))\to H_{n-1}(\gr_{n-2}C_\bullet^{(n-1)}(\eta)\to 0.\end{align*}

We use the simplex trick as in the proof of Proposition~\ref{iddelpprop} to recover the simplicial complex from Appendix~\ref{simpsec}; this time mapping to a subcomplex of $\Delta[n-3]$ ($n-3$ since $\sum_{i=1}^m k_i^\prime=n-1$ after the change of variable and $k_m^\prime\geq 2$.)  \\

By Proposition~\ref{simpprop} the resulting simplicial complex has homology in dimension 0 for $n-3=2,3$ ($H_n(C_\bullet^{(n-1)}(\eta))$ for $n=5,6$), in dimension 1 for $n-3=5,6$ ($H_{n+1}(C_\bullet^{(n-1)}(\eta))$ for $n+1=9,10$), and in dimension 2 for $n-3=8,9$ ($H_{n+2}(C_\bullet^{(n-1)}(\eta))$ for $n+2=13,14$).  \\

We also need consider $n=1,2$ for which $\Delta[n-3]$ does not make sense.  In both of these cases $H_{n+1}(\gr_{n-2}C_\bullet^{(n-1)}(\eta))=0$ and $H_{n}(\gr_{n-1}C_\bullet^{(n-1)}(\eta))$ is one-dimensional, generated by $\eta$ and $(\xi_L)\eta$ respectively.
\end{proof}  

The simplicial correspondence allows us to find representatives of  classes.  We see these are the representatives for the corresponding classes of $H_*(L)$ tensored with $(\xi_L)\eta$ on the right.

$$\begin{array}{|c|l|}\hline
\textrm{location} & \textrm{representative} \\ \hline
H_1^{(1)}(\eta) & \eta \\ \hline
H_2^{(2)}(\eta) & (\xi_L)\eta \sim \eta(\xi) \\ \hline
H_5^{(4)}(\eta) & \eta(\xi)\theta(\xi_L)\eta \\ \hline
H_6^{(5)}(\eta) & \eta(\xi)\theta(\xi_L)^2\eta\sim \eta(\xi)\theta(\xi_L)\eta(\xi)\sim(\xi_L)\eta(\xi)\theta(\xi_L)\eta \\ \hline
H_9^{(7)}(\eta) & \eta(\xi)\theta(\xi_L)\eta(\xi)\theta(\xi_L)\eta \\ \hline
H_{10}^{(8)}(\eta) & \eta(\xi)\theta(\xi_L)\eta(\xi)\theta(\xi_L)^2\eta\sim \eta(\xi)\theta(\xi_L)\eta(\xi)\theta(\xi_L)\eta(\xi)\sim(\xi_L)\eta(\xi)\theta(\xi_L)\eta(\xi)\theta(\xi_L)\eta \\ \hline
H_{13}^{(10)}(\eta) & \eta(\xi)\theta(\xi_L)\eta(\xi)\theta(\xi_L)\eta(\xi)\theta(\xi_L)\eta \\ \hline
H_{14}^{(11)}(\eta) & \eta(\xi)\theta(\xi_L)\eta(\xi)\theta(\xi_L)\eta(\xi)\theta(\xi_L)^2\eta\sim \eta(\xi)\theta(\xi_L)\eta(\xi)\theta(\xi_L)\eta(\xi)\theta(\xi_L)\eta(\xi) \\
& \sim(\xi_L)\eta(\xi)\theta(\xi_L)\eta(\xi)\theta(\xi_L)\eta(\xi)\theta(\xi_L)\eta \\ \hline
 \end{array}$$

\begin{corollary}\label{HHg} 
\begin{align*}
HH^n_{(1-n)}(\eta)&=\left\{\begin{array}{cc}k & \textrm{if $n=1,2$} \\ 0 & \textrm{otherwise,}\end{array}\right.\\
HH^n_{(2-n)}(\eta)&=\left\{\begin{array}{cc} k & \textrm{if $n=5,6$} \\ 0 & \textrm{otherwise} \end{array}\right.\\
HH^n_{(3-n)}(\eta)&=\left\{\begin{array}{cc} k & \textrm{if $n=9,10$} \\ 0 & \textrm{otherwise}\end{array}\right.\\
HH^n_{(4-n)}(\eta)&=\left\{\begin{array}{cc}k & \textrm{if $n=13,14$} \\ 0 & \textrm{otherwise}\end{array}\right.\end{align*} \end{corollary}

\begin{proof}  The complex $C^\bullet_{(1-m)}(\eta)$ is dual to $C_\bullet^{(m)}(\eta)$.  \end{proof}  

\begin{corollary}\label{fprop} 
\begin{align*} 
H_n(C_\bullet^{(n-1)}(\theta))&=\left\{\begin{array}{cc} k & \textrm{if $n=1,2$} \\ 0 & \textrm{otherwise,}\end{array}\right.\\ 
H_n(C_\bullet^{(n-2)}(\theta))&=\left\{\begin{array}{cc}k & \textrm{if $n=5,6$} \\ 0 & \textrm{otherwise,}\end{array}\right.\\ 
H_n(C_\bullet^{(n-3)}(\theta))&=\left\{\begin{array}{cc}k & \textrm{if $n=9,10$} \\ 0 & \textrm{otherwise.}\end{array}\right.\\
H_n(C_\bullet^{(n-4)}(\theta)&=\left\{\begin{array}{cc}k & \textrm{if $n=13,14$} \\ 0 & \textrm{otherwise.}\end{array}\right.\end{align*}\end{corollary}

\begin{proof} There is an isomorphism of complexes $C_\bullet^{(m)}(\eta)\to C_\bullet^{(m-1)}(\theta)$ via $$\eta\mapsto\theta,\theta\mapsto\eta,\xi_L\mapsto\xi,\xi\mapsto\xi_L.$$ \end{proof}

We have representatives,
$$\begin{array}{|c|l|} \hline
\textrm{location} & \textrm{representative} \\ \hline
H_1^{(0)}(\theta) & \theta \\ \hline
H_2^{(1)}(\theta) & \xi\theta\sim \theta\xi_L \\ \hline
H_5^{(3)}(\theta) & \theta(\xi_L)\eta(\xi)\theta \\ \hline
H_6^{(4)}(\theta) & \theta(\xi_L)\eta(\xi)\theta\xi_L\sim\theta(\xi_L)\eta(\xi)^2\theta\sim\xi\theta(\xi_L)\eta(\xi)\theta \\ \hline
H_9^{(6)}(\theta) & \theta(\xi_L)\eta(\xi)\theta(\xi_L)\eta(\xi)\theta \\ \hline
H_{10}^{(7)}(\theta) & \theta(\xi_L)\eta(\xi)\theta(\xi_L)\eta(\xi)\theta\xi_L\sim\theta(\xi_L)\eta(\xi)\theta(\xi_L)\eta(\xi)^2\theta\sim\xi\theta(\xi_L)\eta(\xi)\theta(\xi_L)\eta(\xi)\theta \\ \hline
H_{13}^{(9)}(\theta) & \theta(\xi_L)\eta(\xi)\theta(\xi_L)\eta(\xi)\theta(\xi_L)\eta(\xi)\theta \\ \hline
H_{14}^{(10)}(\theta) & \theta(\xi_L)\eta(\xi)\theta(\xi_L)\eta(\xi)\theta(\xi_L)\eta(\xi)\theta\xi_L\sim\theta(\xi_L)\eta(\xi)\theta(\xi_L)\eta(\xi)\theta(\xi_L)\eta(\xi)^2\theta \\ & \sim\xi\theta(\xi_L)\eta(\xi)\theta(\xi_L)\eta(\xi)\theta(\xi_L)\eta(\xi)\theta \\ \hline
 \end{array}$$

\begin{corollary}\label{HHf}
\begin{align*}
HH^n_{(1-n)}(\theta)&=\left\{\begin{array}{cc}k & \textrm{if $n=1,2$} \\ 0 & \textrm{otherwise,}\end{array}\right.\\
HH^n_{(2-n)}(\theta)&=\left\{\begin{array}{cc} k & \textrm{if $n=5,6$} \\ 0 & \textrm{otherwise} \end{array}\right.\\
HH^n_{(3-n)}(\theta)&=\left\{\begin{array}{cc}k & \textrm{if $n=9,10$} \\ 0 & \textrm{otherwise}\end{array}\right.\\
HH^n_{(4-n)}(\theta)&=\left\{\begin{array}{cc}k & \textrm{if $n=13,14$} \\ 0 & \textrm{otherwise}\end{array}\right.\end{align*} \end{corollary}

\begin{proof}\end{proof}

\subsection{Cohomology with coefficients in $B_0$ and $B_1$}
\label{AzeroAone}

We give $B_0$ the structure of a $B$-bimodule by identifying it with $B/B_1$.  We denote $$HH^n_{(m)}(B,B_1):=HH^n_{(m)}(B_1), \ \ HH^n_{(m)}(B,B_0):=HH^n_{(m)}(B_0).$$ The short exact sequences of $B$-bimodules $$0\to (\xi,\xi_L)\to B_1\to B_1/(\xi,\xi_L)=\langle\eta\rangle\to 0$$ and $$0\to B_+/B_1=\langle\theta\rangle\to B/B_1=B_0\to B/B_+=\langle\id_L,\id_{\oc}\rangle\to 0$$ give rise to long exact sequences in Hochschild cohomology. \\

\begin{prop}\label{AzeroAoneprop}\mbox{}
 \begin{enumerate}
\item[a.] $HH^1_{(0)}(B_1)=k$ and $HH^1_{(0)}(B_0)=k$ \\
\item[b.] $HH^2_{(-1)}(B_1)=0$ and $HH^3_{(-1)}(B_1)=k.$\\
\item[c.] $HH^2_{(-1)}(B_0)=k$, $HH^3_{(-2)}(B_0)=k^2$, $HH^4_{(-2)}(B_1)=k^2,$ and $HH^5_{(-3)}(B_1)=k.$\\
\item[d.] $HH^4_{(-3)}(B_0)=k$ and $HH^5_{(-3)}(B_0)=0$. \\
\item[e.] $HH^6_{(-4)}(B_1)=0$ and $HH^7_{(-4)}(B_1)=k.$ \\
\item[f.] $HH^6_{(-4)}(B_0)=k,$ $HH^7_{(-5)}(B_0)=k^2,$ $HH^8_{(-5)}(B_1)=k^2,$ and $HH^9_{(-6)}(B_1)=k.$ \\
\item[g.] $HH^8_{(-6)}(B_0)=k$ and $H^9_{(-6)}(B_0)=0.$ \\
\item[h.] $HH^{10}_{(-7)}(B_1)=0$ and $HH^{11}_{(-7)}(B_1)=k.$ \\
\item[i.] $HH^{10}_{(-7)}(B_0)=k,$ $HH^{11}_{(-8)}(B_0)=k^2,$ $HH^{12}_{(-8)}(B_1)=k^2,$ and $HH^{13}_{(-9)}(B_1)=k.$\\
\item[j.] $HH^{12}_{(-9)}(B_0)=k$ and $HH^{13}_{(-9)}(B_0)=0.$ \\
\end{enumerate}\end{prop}

\begin{proof}\mbox{}
\begin{enumerate}
\item[a.] These follow from the exact sequences $$0\to HH^1_{(0)}(B_1)\to HH^1_{(0)}(\eta)=k\to 0,$$ $$0\to HH^1_{(0)}(\theta)=k\to HH^1_{(0)}(B_0)\to0.$$ \\

\item[b.] We have the exact sequence $$0\to HH^2_{(-1)}(B_1)\to HH^2_{(-1)}(\eta)=k\xrightarrow{\delta} HH^3_{(-1)}(\xi_L,\xi)=k^2\to HH^3_{(-1)}(B_1)\to0.$$ \\
The class $\phi$ of $HH^2_{(-1)}(\eta)$ is dual to the homology class $[(\xi_L)\eta]=[\eta(\xi)]$, so we may assume $\phi((\xi_L)\eta)=\phi(\eta(\xi))=\eta$. \\

The classes of $HH^3_{(-1)}(\xi_L,\xi)$ are dual to the homology classes $[\theta(\xi_L)\eta]$ and $[\eta(\xi)\theta]$.  We check,
\begin{align*}
\delta(\phi)(\theta(\xi_L)\eta) & = \theta\phi((\xi_L)\eta)=\xi, \\
\delta(\phi)(\eta(\xi)\theta) & = -\phi(\eta(\xi))f\\
\theta&=-\xi_L.
\end{align*}
So $\delta$ is injective. \\

\item[c.] The result is obvious from the sequences: \\
\begin{align*}
&0\to HH^2_{(-1)}(\theta)=k\to HH^2_{(-1)}(B_0)\to0, \\
&0\to HH^3_{(-2)}(B_0)\to HH^3_{(-2)}(\id_{L},\id_{\oc})=k^2\to0,\\  
&0\to HH^4_{(-2)}(\xi_L,\xi)=k^2\to HH^4_{(-2)}(B_1)\to 0,\\
&0\to HH^5_{(-3)}(B_1)\to HH^5_{(-3)}(\eta)=k\to0.\end{align*}

\item[d.] We consider the sequence, $$0\to HH^4_{(-3)}(B_0)\to HH^4_{(-3)}(\id_{L},\id_{\oc})=k^2\xrightarrow{\delta} HH^5_{(-3)}(\theta)=k\to HH^5_{(-3)}(B_0)\to0.$$ \\
Representatives $\phi_1,\phi_2$ of the classes $HH^4_{(-3)}(\id_{L},\id_{\oc})$ are dual to the classes $[\theta(\xi_L)\eta(\xi)]$ and $[(\xi_L)\eta(\xi)\theta]$, so we assume $\phi_1(\theta(\xi_L)\eta(\xi))=\id_{\oc}$ and $\phi_2(\eta(\xi)\theta\xi_L)=\id_{L}$.  The representative of $HH^5_{(-3)}(\theta)$ is dual to the class $[\theta(\xi_L)\eta(\xi)\theta]$, so we check
\begin{align*}
\delta(\phi_1)(\theta(\xi_L)\eta(\xi)\theta) & = -\phi_1(\theta(\xi_L)\eta(\xi))\theta=-\theta, \\
\delta(\phi_2)(\theta(\xi_L)\eta(\xi)\theta) & =  \theta\phi_2(\xi_L \eta(\xi)\theta)=\theta.
\end{align*}
So the kernel is generated by $\phi_1+\phi_2$. \\

\item[e.] We look at the exact sequence $$0\to HH^6_{(-4)}(B_1)\to HH^6_{(-4)}(\eta)=k\xrightarrow{\delta} HH^7_{(-4)}(\xi_L,\xi)=k^2\to HH^7_{(-4)}(B_1)\to0.$$ \\
The class $\phi$ of $HH^6_{(-4)}(\eta)$ is dual to $[\eta(\xi)\theta(\xi_L)\eta(\xi)]=[(\xi_L)\eta(\xi)\theta(\xi_L)\eta]$, so we assume that $\phi$ at each of these gives $\eta$.  The classes of $HH^7_{(-4)}(\xi_L,\xi)$ are dual to $[\eta(\xi)\theta(\xi_L)\eta(\xi)\theta]$ and $[\theta(\xi_L)\eta(\xi)\theta(\xi_L)\eta]$, so we check
\begin{align*}
\delta(\phi)(\eta(\xi)\theta(\xi_L)\eta(\xi)\theta) & =  -\phi(\eta(\xi)\theta(\xi_L)\eta(\xi))\eta=-\eta\theta=-\xi_L, \\
\delta(\phi)(\theta(\xi_L)\eta(\xi)\theta(\xi_L)\eta) & =  \theta\phi((\xi_L)\eta(\xi)\theta(\xi_L)\eta)=\theta\eta=\xi.
\end{align*}
So $\delta$ is injective. \\

\item[f.] We have exact sequences, \\
\begin{align*}&0\to HH^6_{(-4)}(\theta)=k\to HH^6_{(-4)}(B_0)\to0,\\
&0\to HH^7_{(-5)}(B_0)\to HH^7_{(-5)}(\id_{L},\id_{\oc})=k^2\to0,\\ 
&0\to HH^8_{(-5)}(\xi_L,\xi)=k^2\to HH^8_{(-5)}(B_1)\to0,\\
&0\to HH^9_{(-6)}(B_1)\to HH^9_{(-6)}(\eta)=k\to0.\end{align*}

\item[g.] This concerns the exact sequence $$0\to HH^8_{(-6)}(B_0)\to HH^8_{(-6)}(\id_{L},\id_{\oc})=k^2\xrightarrow{\delta} HH^9_{(-6)}(\theta)=k\to HH^9_{(-6)}(B_0)\to0$$ The classes of $HH^8_{(-6)}(\id_{L},\id_{\oc})$ are dual to $[\theta(\xi_L)\eta(\xi)\theta(\xi_L)\eta(\xi)]$ and $[\eta(\xi)\theta(\xi_L)\eta(\xi)\theta\xi_L]$, thus representatives $\phi_1$ and $\phi_2$ can be assumed to satisfy 
\begin{align*}
\phi_1(\theta(\xi_L)\eta(\xi)\theta(\xi_L)\eta(\xi))&=\id_{\oc}, \\ \phi_2((\xi_L)\eta(\xi)\theta(\xi_L)\eta(\xi)\theta)&=\id_{L}.\end{align*}
The class in $HH^9_{(-6)}(\theta)$ is dual to $[\theta(\xi_L)\eta(\xi)\theta(\xi_L)\eta(\xi)\theta]$ so we evaluate
\begin{align*}
\delta(\phi_1)(\theta(\xi_L)\eta(\xi)\theta(\xi_L)\eta(\xi)\theta) & = \phi_1(\theta(\xi_L)\eta(\xi)\theta(\xi_L)\eta(\xi))\theta=-\theta, \\
\delta(\phi_2)(\theta(\xi_L)\eta(\xi)\theta(\xi_L)\eta(\xi)\theta) & =  \theta\phi_2((\xi_L)\eta(\xi)\theta(\xi_L)\eta(\xi)\theta)=\theta.
\end{align*}
Thus the kernel of $\delta$ is $\phi_1+\phi_2$. \\

\item[h.] We consider the exact sequence $$0\to HH^{10}_{(-7)}(B_1)\to HH^{10}_{(-7)}(\eta)=k\xrightarrow{\delta}HH^{11}_{(-7)}(\xi_L,\xi)=k^2\to HH^{11}_{(-7)}(B_1)\to0.$$ \\
The class in $HH^{10}_{(-7)}(\eta)$ is dual to $$[\eta(\xi)\theta(\xi_L)\eta(\xi)\theta(\xi_L)\eta(\xi)]=[(\xi_L)\eta(\xi)\theta(\xi_L)\eta(\xi)\theta(\xi_L)\eta],$$ so we assume a representative $\phi$ evaluated there returns $\eta$.  The classes in $HH^{11}_{(-7)}(\xi_L,\xi)$ are dual to $$[\eta(\xi)\theta(\xi_L)\eta(\xi)\theta(\xi_L)\eta(\xi)\theta] \ \textrm{and} \ [\theta(\xi_L)\eta(\xi)\theta(\xi_L)\eta(\xi)\theta(\xi_L)\eta].$$  We check,
\begin{align*}
\delta(\phi)(\eta(\xi)\theta(\xi_L)\eta(\xi)\theta(\xi_L)\eta(\xi)\theta) & =  -\phi(\eta(\xi)\theta(\xi_L)\eta(\xi)\theta(\xi_L)\eta(\xi))\theta=-\xi_L, \\
\delta(\phi)(\theta(\xi_L)\eta(\xi)\theta(\xi_L)\eta(\xi)\theta(\xi_L)\eta) & =  \theta\phi((\xi_L)\eta(\xi)\theta(\xi_L)\eta(\xi)\theta(\xi_L)\eta)=\xi.
\end{align*}
So $\delta$ is injective. \\

\item[i.] We have sequences,
\begin{align*}
&0\to HH^{10}_{(-7)}(\theta)=k\to HH^{10}_{(-7)}(B_0)\to 0,\\
&0\to HH^{11}_{(-8)}(B_0)\to HH^{11}_{(-8)}(\id_{L},\id_{\oc})=k^2\to0,\\
&0\to HH^{12}_{(-8)}(\xi_L,\xi)=k^2\to HH^{12}_{(-8)}(B_1)\to 0,\\
&0\to HH^{13}_{(-9)}(B_1)\to HH^{13}_{(-9)}(\eta)=k\to 0. \end{align*}\\

\item[j.] We consider the sequence, $$0\to HH^{12}_{(-9)}(B_0)\to HH^{12}_{(-9)}(\id_{L},\id_{\oc})=k^2\xrightarrow{\delta}HH^{13}_{(-9)}(\theta)=k\to HH^{13}_{(-9)}(B_0)\to0.$$  The classes in $HH^{12}_{(-9)}(\id_{L},\id_{\oc})$ are represented by $\phi_1$ and $\phi_2$ dual to $$[(\xi_L)\eta(\xi)\theta(\xi_L)\eta(\xi)\theta\xi_L \eta(\xi)\theta] \ \textrm{and} \ [\theta(\xi_L)\eta(\xi)\theta(\xi_L)\eta(\xi)\theta(\xi_L)\eta(\xi)],$$ and the class in $HH^{13}_{(-9)}(\theta)$ is dual to $$[\theta(\xi_L)\eta(\xi)\theta(\xi_L)\eta(\xi)\theta(\xi_L)\eta(\xi)\theta].$$  We check,
\begin{align*}
\delta(\phi_1)(\theta(\xi_L)\eta(\xi)\theta(\xi_L)\eta(\xi)\theta(\xi_L)\eta(\xi)\theta) & =  \theta\phi_1((\xi_L)\eta(\xi)\theta(\xi_L)\eta(\xi)\theta(\xi_L)\eta(\xi)\theta)=\theta, \\
\delta(\phi_2)(\theta(\xi_L)\eta(\xi)\theta(\xi_L)\eta(\xi)\theta(\xi_L)\eta(\xi)\theta) & =  -\phi_2(\theta(\xi_L)\eta(\xi)\theta(\xi_L)\eta(\xi)\theta(\xi_L)\eta(\xi))\theta=-\theta.
\end{align*}
So the kernel of $\delta$ is $\phi_1+\phi_2$.
\end{enumerate}
\end{proof}

\subsection{Hochschild cohomology with coefficients in $B$}
\label{Ahoch}

Denote $HH^n_{(m)}(B,B):=HH^n_{(m)}(B)$.  We use the long exact sequence on cohomology induced by the short exact sequence of $B$-bimodules,
$$0\to B_1\to B\to B_0\to 0.$$

\begin{prop}\label{HHAB}\mbox{}
\begin{enumerate}
\item[a.] $HH^1_{(0)}(B)=k^2$ \\
\item[b.] $HH^2_{(-1)}(B)=k$ and $HH^3_{(-1)}(B)=k$ \\
\item[c.] $HH^3_{(-2)}(B)=0$ and $HH^4_{(-2)}(B)=0$ \\
\item[d.] $HH^4_{(-3)}(B)=k$ and $HH^5_{(-3)}(B)=k$ \\
\item[e.] $HH^6_{(-4)}(B)=k$ and $HH^7_{(-4)}(B)=k$ \\
\item[f.] $HH^7_{(-5)}(B)=0$ and $HH^8_{(-5)}(B)=0$ \\
\item[g.] $HH^8_{(-6)}(B)=k$ and $HH^9_{(-6)}(B)=k$ \\
\item[h.] $HH^{10}_{(-7)}(B)=k$ and $HH^{11}_{(-7)}(B)=k$ \\
\item[i.] $HH^{11}_{(-8)}(B)=0$ and $HH^{12}_{(-8)}(B)=0$ \\
\item[j.] $HH^{12}_{(-9)}(B)=k$ and $HH^{12}_{(-9)}(B)=k$
\end{enumerate} \end{prop}

\begin{proof}\mbox{}
\begin{enumerate}
\item[a.] Follows from part Proposition~\ref{AzeroAoneprop}(a) and the sequence $$0\to HH^1_{(0)}(B_1)=k\to HH^1_{(0)}(B)\to HH^1_{(0)}(B_0)=k\to0.$$ \\

\item[b.] We consider the sequence
$$0\to HH^2_{(-1)}(B)\to HH^2_{(-1)}(B_0)=k\xrightarrow{\delta} HH^3_{(-1)}(B_1)=k\to HH^3_{(-1)}(B)\to0.$$ A representative $\psi$ of $HH^2_{(-1)}(B_0)$ is dual to $[\xi\theta]=[\theta\xi_L]$.  We evaluate 
\begin{align*}
\delta(\psi)(\theta(\xi_L)\eta) & =  -\psi(\theta\xi_L)\eta=-\xi, \\
\delta(\psi)(\eta(\xi)\theta) & =  \eta\psi(\xi\theta)=\xi_L.
\end{align*}
So the image of $\psi$ under $\delta$ is the negative of the image of $\phi$ from Proposition~\ref{AzeroAoneprop}(b), but this was made 0 already there.  Thus $\delta$ is zero here. \\

\item[c.] We have the sequence
$$0\to HH^3_{(-2)}(B)\to HH^3_{(-2)}(B_0)=k^2\xrightarrow{\delta} HH^4_{(-2)}(B_1)=k^2\to HH^4_{(-2)}(B)\to0$$  Representatives $\alpha_1,\alpha_2$ generating $HH^3_{(-2)}(B_0)$ are dual to $[\eta(\xi)\theta]$ and $[\theta(\xi_L)\eta]$.  Representatives $\beta_1,\beta_2$ generating $HH^4_{(-2)}(B_1)$ are dual to $[\eta(\xi)\theta\xi_L]$ and $[\theta(\xi_L)\eta(\xi)].$  Thus we assume that $\alpha_1(\eta(\xi)\theta)=\id_{L}$ and $\alpha_2(\theta(\xi_L)\eta)=\id_{\oc}$ and evaluate
\begin{align*}
\delta(\alpha_1)(\eta(\xi)\theta\xi_L) & =  -\alpha_1(\eta(\xi)\theta)\xi_L=-\xi_L, \\
\delta(\alpha_1)(\theta(\xi_L)\eta(\xi)) & =  0, \\
\delta(\alpha_2)(\eta(\xi)\theta\xi_L) & =  0, \\
\delta(\alpha_2)(\theta(\xi_L)\eta(\xi)) & = -\alpha_2(\theta(\xi_L)\eta)\xi=\xi.
\end{align*}
So here $\delta$ is an isomorphism. \\

\item[d.] We have a sequence $$0\to HH^4_{(-3)}(B)\to HH^4_{(-3)}(B_0)=k\xrightarrow{\delta} HH^5_{(-3)}(B_1)=k\to HH^5_{(-3)}(B)\to0.$$  The class $\psi$ generating $HH^4_{(-3)}(B_0)$ is $\phi_1+\phi_2$ from Proposition~\ref{AzeroAoneprop}(d).  The class generating $HH^5_{(-3)}(B_1)$ is dual to the class $[\eta(\xi)\theta(\xi_L)\eta]$.  Then we have
\begin{align*}
\delta(\phi_1)(\eta(\xi)\theta(\xi_L)\eta) & =  \eta\phi_1(\xi\theta(\xi_L)\eta)=\eta, \\
\delta(\phi_2)(\eta(\xi)\theta(\xi_L)\eta) & =  -\phi(\eta(\xi)\theta\xi_L)\eta=-\eta.
\end{align*}
Thus again as in the proof of Proposition~\ref{AzeroAoneprop}(d) the kernel is $\phi_1+\phi_2$, meaning that in this case $\delta=0$.  \\

\item[e.] We have a sequence $$0\to HH^6_{(-4)}(B)\to HH^6_{(-4)}(B_0)=k\xrightarrow HH^7_{(-4)}(B_1)=k\to HH^7_{(-4)}(B)\to0.$$  Let $\psi$ be a generator by $HH^6_{(-4)}(B_0)$ dual to $[\theta(\xi_L)\eta(\xi)\theta\xi_L]=[\xi\theta(\xi_L)\eta(\xi)\theta]$.  Following Proposition~\ref{AzeroAoneprop}(e) we have,
\begin{align*}
\delta(\psi)(\eta(\xi)\theta(\xi_L)\eta(\xi)\theta) & =  \eta\psi(\xi\theta(\xi_L)\eta(\xi)\theta)=\xi_L, \\
\delta(\psi)(\theta(\xi_L)\eta(\xi)\theta(\xi_L)\eta) & =  -\psi(\theta(\xi_L)\eta(\xi)\theta\xi_L)\eta=-\xi.
\end{align*}
So the image of $\psi$ under $\delta$ is the negative of the image of $\phi$ from Proposition~\ref{AzeroAoneprop}(e), but this was made 0 already there.  Thus $\delta=0$ here. \\

\item[f.] We have a sequence $$0\to HH^7_{(-5)}(B)\to HH^7_{(-5)}(B_0)=k^2\xrightarrow HH^8_{(-5)}(B_1)=k^2\to HH^8_{(-5)}(B)\to0.$$  The classes $\alpha_1,\alpha_2$ generating $HH^7_{(-5)}(B_0)$ are dual to $$[\eta(\xi)\theta(\xi_L)\eta(\xi)\theta] \ \textrm{and} \ [\theta(\xi_L)\eta(\xi)\theta(\xi_L)\eta],$$ respectively. The classes generating $HH^8_{(-5)}(B_1)$ are dual to $$[\theta(\xi_L)\eta(\xi)\theta(\xi_L)\eta(\xi)] \ \textrm{and} \ [\eta(\xi)\theta(\xi_L)\eta(\xi)\theta(\xi_L)].$$  So we check
\begin{align*}
\delta(\alpha_1)(\theta(\xi_L)\eta(\xi)\theta(\xi_L)\eta(\xi)) & =  0, \\
\delta(\alpha_1)(\eta(\xi)\theta(\xi_L)\eta(\xi)\theta(\xi_L)) & =  \alpha_1(\eta(\xi)\theta(\xi_L)\eta(\xi)\theta)\xi_L=\xi_L, \\
\delta(\alpha_2)(\theta(\xi_L)\eta(\xi)\theta(\xi_L)\eta(\xi)) & =  \alpha_2(\theta(\xi_L)\eta(\xi)\theta(\xi_L)\eta)\xi=\xi, \\
\delta(\alpha_2)(\eta(\xi)\theta(\xi_L)\eta(\xi)\theta\xi_L) & =  0.
\end{align*}
So $\delta$ is an isomorphism. \\

\item[g.] We have a sequence $$0\to HH^8_{(-6)}(B)\to HH^8_{(-6)}(B_0)=k\xrightarrow HH^9_{(-6)}(B_1)=k\to HH^9_{(-6)}(B)\to0.$$  The generator of $HH^8_{(-6)}(B_0)$ is $\phi_1+\phi_2$ from Proposition~\ref{AzeroAoneprop}(g) and from $HH^9_{(-6)}(B_1)$ is dual to $[\eta(\xi)\theta(\xi_L)\eta(\xi)\theta(\xi_L)\eta]$.  So we check,
\begin{align*}
\delta(\phi_1)(\eta(\xi)\theta(\xi_L)\eta(\xi)\theta(\xi_L)\eta) & =  \eta\phi_1(\xi\theta(\xi_L)\eta(\xi)\theta(\xi_L)\eta)=\eta, \\
\delta(\phi_2)(\eta(\xi)\theta(\xi_L)\eta(\xi)\theta(\xi_L)\eta) & =  -\phi_2(\eta(\xi)\theta(\xi_L)\eta(\xi)\theta\xi_L)\eta=-\eta.
\end{align*}
Thus again the kernel is $\phi_1+\phi_2$ so here $\delta=0$. \\

\item[h.] We have a sequence $$0\to HH^{10}_{(-7)}(B)\to HH^{10}_{(-7)}(B_0)=k\xrightarrow{\delta}HH^{11}_{(-7)}(B_1)=k\to HH^{11}_{(-7)}(B)\to0.$$  For precisely the reasons as in (b,e) above $\delta=0$. \\

\item[i.] We have a sequence 
\begin{align*}&HH^{11}_{(-8)}(B_1)=0\to HH^{11}_{(-8)}(B)\to HH^{11}_{(-8)}(B_0)=k^2\xrightarrow{\delta}\\  &HH^{12}_{(-8)}(B_1)=k^2\to HH^{12}_{(-8)}(B)\to HH^{12}_{(-8)}(B_0)=0.\end{align*}  For precisely the reasons as in (c,f) above $\delta$ is an isomorphism. \\

\item[j.] We have a sequence \begin{align*}&HH^{12}_{(-9)}(B_1)=0\to HH^{12}_{(-9)}(B)\to HH^{12}_{(-9)}(B_0)=k\xrightarrow{\delta} \\ &HH^{13}_{(-9)}(B_1)=k\to HH^{13}_{(-9)}(B)\to HH^{13}_{(-9)}(B_0)=0.\end{align*}  For precisely the reasons as in (d,g) above $\delta=0.$ 
\end{enumerate}
\end{proof}

\begin{theorem}
\label{HHA}
The Hochschild cohomology of $B$ for maps of internal degrees $1-n,$ $2-n$, and $3-n$ is
\begin{align*}
HH^n_{(1-n)}(B)&=\left\{\begin{array}{cc}k^2 & \textrm{if $n=1$} \\ k & \textrm{if $n=2,4$} \\ 0 & \textrm{otherwise.}\end{array}\right. \\
HH^n_{(2-n)}(B)&=\left\{\begin{array}{cc}k & \textrm{if $n=3,5,6,8$} \\ 0 & \textrm{otherwise.}\end{array}\right. \\
HH^n_{(3-n)}(B)&=\left\{\begin{array}{cc}k & \textrm{if $n=7,9,10,12$} \\ 0 & \textrm{otherwise.}\end{array}\right.\end{align*}
\end{theorem}

\section{The $A_\infty$-algebra of a complex elliptic curve}
\label{Ainfalg}

\subsection{$A_\infty$-structure on $\Ext^*(\oc\oplus L,\oc\oplus L)$}
\label{amend}
Let $C=\C/(\Z+\tau\Z)$ be a complex elliptic curve (in particular $\operatorname{Im}(\tau)>0$).  Below we write $m_n$ in place of the more appropriate $m_n(\tau)$, as the operation $m_n$ depends on $\tau$, but we now consider it fixed. In \cite{Pol-E} Theorem 2.5.1, Polishchuk proves

\begin{theorem}
\label{Sasha}
The only non-trivial higher products $m_n$ of the $A_\infty$-structure on \\ $B=\Ext^*(\oc\oplus L,\oc\oplus L)$ are of the form

\begin{align*}m_n((\xi)^a\theta(\xi_L)^b\eta(\xi)^c\theta(\xi_L)^d)&=M(a,b,c,d)\cdot\theta,\\
m_n((\xi_L)^a\eta(\xi)^b\theta(\xi_L)^c\eta(\xi)^d)&=M(a,b,c,d)\cdot\eta,\\
m_n((\xi)^a\theta(\xi_L)^b\eta(\xi)^c\theta(\xi_L)^d\eta(\xi)^e)&=M(a+e+1,b,c,d)\cdot\id_{\oc},\\
m_n((\xi_L)^a\eta(\xi)^b\theta(\xi_L)^c\eta(\xi)^d\theta(\xi_L)^e)&=M(a+e+1,b,c,d)\cdot\id_L.\end{align*}  All products $m_n$ with odd $n$ vanish.\end{theorem}

We set about amending this structure based on the result in Theorem~\ref{HHA} and the discussion in Section~\ref{HCAinf}. \\

The associative algebra $B$ obtained by restricting the $A_\infty$-structure from Theorem~\ref{Sasha} is independent of $\tau$.  The calculation that $m_3=0$ also does not depend on $\tau$. 
Theorem~\ref{Sasha} has a nonzero $m_4$, but the calculation in Theorem~\ref{HHA} that $HH^4_{(-2)}(B)=0$ implies that all possible $m_4$ extending $m_1=0,m_2,m_3=0$ are equivalent to $m_4=0$ (since $m_3=m_4=m_5=\cdots=0$ is one possible $A_\infty$-structure on $B$).  The discussion in Section~\ref{HCAinf} therefore implies that there is a strict equivalence $f:B\to B$ such that $f*m_4=0$. \\

\begin{prop}
\label{equiv}
The equivalance $(f_n):B\to B$ defined by $f_1=\id_{B}$, \\ $f_2=f_4=f_5=\cdots=0,$ 
\begin{align*}
f_3&=M(1,0,0,0)[([\eta(\xi)^2]^*-[(\xi_L)^2\eta]^*-[(\xi_L)\eta(\xi)]^*)\otimes \eta+([\theta(\xi_L)^2]^*-[(\xi)^2\theta]^*-[\xi\theta\xi_L]^*)\otimes \theta+ \\ &([\xi\theta\eta]^*+[\theta(\xi_L)\eta]^*-[\theta\eta(\xi)]^*)\otimes\id_{\oc}+([(\xi_L)\eta\theta]^*+[\eta(\xi)\theta]^*-[\eta\theta\xi_L]^*)\otimes\id_L]\end{align*} defines a new $A_\infty$-structure $(m_n^\prime)$ on $B$ with \\
\begin{enumerate}
\item $m_4^\prime=0$; \\
\item $m_k^\prime=0$ for $k$ odd; and \\
\item $m_6^\prime,m_8^\prime$ Hochschild cocycles. 
\end{enumerate} \end{prop}

\begin{proof} 
\begin{enumerate}
\item Recall (see \cite{Kel}) that for $(f_n)$ to be an equivalence between $(B,m_n)$ and $(B,m_n^\prime)$, it must satisfy the relations, for all $k\geq 1$,
$$\sum (-1)^{r+st}f_u(1^{\otimes r}\otimes m_s\otimes 1^{\otimes t})=\sum(-1)^sm_r^\prime(f_{i_1}\otimes f_{i_2}\otimes\cdots\otimes f_{i_r})$$ where the first sum runs over all decompositions $k=r+s+t$ with $u=r+1+t$, and the second sum runs over all $1\leq r\leq k$ and all decompositions $k=i_1+\cdots+i_r$ and $$s=(r-1)(i_1-1)+(r-2)(i_2-1)+\cdots+2(i_{r-2}-1)+(i_{r-1}-1).$$  Since $f_1=\id_{B}$ and $f_2=0$ it follows that $m_3^\prime=m_3=0$, and in order for $m_4^\prime=0$ we must have $\delta f_3=m_4$.  That this is true is verified by direct calculation. \\

\item We show by induction that $m_{2k+1}^\prime=0$ for $k\in \Z_{\geq 1}$.  The base case is above, and for $k\geq 2$ the left side of the relation vanishes.  This is clear since $f_i=0$ except for $i=1,3$, and the terms involving $f_1,f_3$ have $m_{2k+1}$ and $m_{2k-1}$, respectively, which are both zero.  The right side vanishes except for the term $m_{2k+1}^\prime$, since in order for $i_1+\cdots+i_r=2k+1$ where all $i_j$ are 1 or 3, we must have $r$ odd.  The induction hypothesis guarantees that in these cases $m_r^\prime=0$. \\

\item Finally, the $A_\infty$-relations for $(m_n^\prime)$ give\begin{align*}
\delta m_6^\prime &=\Phi_6(m_3^\prime,m_4^\prime,m_5^\prime), \\ \delta m_8^\prime&=\Phi_8(m_3^\prime,m_4^\prime,m_5^\prime,m_6^\prime,m_7^\prime),\end{align*}
where $\Phi_k$ is a quadratic expression.  The right side of the first must be zero since $m_3^\prime=m_4^\prime=m_5^\prime=0$.  The right side of the second must be zero since only $m_6^\prime\neq 0$ and in this quadratic expression, $m_6^\prime$ is paired with $m_4^\prime$. \end{enumerate}\end{proof}

We determine from the relation in the proof of (1) that
\begin{align*}
m_6^\prime & =  m_6+f_3(1^{\otimes 2}\otimes m_4-1\otimes m_4\otimes 1+m_4\otimes 1^{\otimes 2})-m_2(f_3\otimes f_3), \\
m_8^\prime & =  m_8+f_3(1^{\otimes 2}\otimes m_6-1\otimes m_6\otimes 1+m_6\otimes 1^{\otimes 2})-m_6^\prime\left(\sum_{r+1+t=6}1^{\otimes r}\otimes f_3\otimes 1^{\otimes t}\right).
\end{align*}

\section{Recovering the $j$-invariant of $C$}
\label{jinvariant}

We set about constructing an explicit isomorphism $HH^6_{(-4)}(B,B)\oplus HH^8_{(-6)}(B,B)\to \C^2$.

Let $S=\{\id_{\oc},\id_L,\theta,\eta,\xi,\xi_L\}$, our usual basis for $B$.  Then for every $x\in A^{\otimes n}$ and $y\in S$ we define an $R$-linear function $t_x^y:\Hom_R^{(-4)}(A^{\otimes n},A)\to \C$ by $t_x^y(\phi)$ is the coefficient on $y$ in $\phi(x)$. \\

\begin{prop} 
\label{m6prop} 
Let \begin{align*}
x&=-\eta\theta\eta(\xi)\theta(\xi_L)-(\xi_L)\eta(\xi)\theta\eta\theta+(\xi_L)\eta\theta\eta\theta(\xi_L)+\eta\theta\eta\theta(\xi_L)^2-\eta\theta(\xi_L)^2\eta\theta+(\xi_L)^2\eta\theta\eta\theta-\\&\eta(\xi)\theta\eta\theta(\xi_L)-(\xi_L)\eta\theta\eta(\xi)\theta.\end{align*}  
\begin{enumerate}
\item If $\phi\in\Hom_R^{(-4)}(B_+^{\otimes 6},B)$ is a Hochschild cocycle, then $$\beta(\phi)=(t_x^{\id_L}-t_{\theta(\xi_L)\eta(\xi)\theta(\xi_L)}^\theta)(\phi)=0$$ if and only if $\phi$ is a coboundary.  Thus $\beta$ is constant on cohomology classes.
\item Let $c\in HH^6_{(-4)}(B,B)$ such that $\beta(c)=-5$.  Then the $\C$-linear function \\ $\alpha:HH^6_{(-4)}(B,B)\to\C$ defined by $c\mapsto 1$ is an isomorphism such that $m_6^\prime\mapsto  t^4e_4.$ \end{enumerate}
\end{prop} 

\begin{proof} For $\psi\in\Hom_R^{(-4)}(B_+^{\otimes 5},B)$, where $\delta$ is the Hochschild differential, we have
$\delta\psi(x)=\alpha_1+\alpha_2+\alpha_3$ where
\begin{align*}
\alpha_1 & = -\eta\psi(\theta\eta(\xi)\theta\xi_L)-\xi_L\psi(\eta(\xi)\theta\eta\theta)+\xi_L\psi(\eta\theta\eta\theta\xi_L)+\eta\psi(\theta\eta\theta(\xi_L)^2)-\eta\psi(\theta(\xi_L)^2\eta\theta)+\\ &  \xi_L\psi((\xi_L)\eta\theta\eta\theta)-\eta\psi(\xi\theta\eta\theta\xi_L)-\xi_L\psi(\eta\theta\eta(\xi)\theta), \\
\alpha_2 & = -\psi(-(\xi_L)\eta(\xi)\theta\xi_L+(\xi_L)\eta\theta(\xi_L)^2-(\xi_L)^3\eta\theta)+\psi(-\eta(\xi)^2\theta\xi_L+(\xi_L)^2\eta\theta\xi_L+\\ &  \eta(\xi)\theta(\xi_L)^2-(\xi_L)^2\eta(\xi)\theta)-\psi((\xi_L)\eta(\xi)\theta\xi_L+\eta\theta(\xi_L)^3+(\xi_L)^3\eta\theta-\eta(\xi)^2\theta\xi_L-\\ &  (\xi_L)\eta(\xi)^2\theta)+\psi(-(\xi_L)\eta(\xi)^2\theta+(\xi_L)\eta\theta(\xi_L)^2+(\xi_L)^2\eta(\xi)\theta-\eta(\xi)\theta(\xi_L)^2)-\\ &  \psi(-(\xi_L)\eta(\xi)\theta\xi_L-\eta\theta(\xi_L)^3+(\xi_L)^2\eta\theta\xi_L)\\
& =  \psi((\xi_L)\eta(\xi)\theta\xi_L)=t_{(\xi_L)\eta(\xi)\theta\xi_L}^{\id_L}(\psi)\id_L, \\
\alpha_3 & =  \psi(\eta\theta\eta(\xi)\theta)\xi_L+\psi((\xi_L)\eta(\xi)\theta\eta)\theta-\psi((\xi_L)\eta\theta\eta\theta)\xi_L-\psi(\eta\theta\eta\theta\xi_L)\xi_L+\psi(\eta\theta(\xi_L)^2\eta)\theta-\\ &  \psi((\xi_L)^2\eta\theta\eta)\theta+\psi(\eta(\xi)\theta\eta\theta)\xi_L+\psi((\xi_L)\eta\theta\eta(\xi))\theta.
\end{align*}
We have $\alpha_1=\alpha_3=0$ for degree reasons, so in fact $$\delta\psi(x)=\alpha_2=t_{(\xi_L)\eta(\xi)\theta\xi_L}^{\id_L}\id_L,$$ thus $$t_x^{\id_L}(\delta\psi)=t_{(\xi_L)\eta(\xi)\theta\xi_L}^{\id_L}(\psi).$$

But also
\begin{align*}
\delta\psi(\theta(\xi_L)\eta(\xi)\theta\xi_L) & =  \theta\psi((\xi_L)\eta(\xi)\theta\xi_L) \\
& =  \theta\cdot t_{(\xi_L)\eta(\xi)\theta\xi_L}^{\id_L}(\psi)\id_L \\
& =  t_{(\xi_L)\eta(\xi)\theta\xi_L}^{\id_L}(\psi)\theta,
\end{align*}
so $$t_{\theta(\xi_L)\eta(\xi)\theta\xi_L}^\theta(\delta\psi)=t_{(\xi_L)\eta(\xi)\theta\xi_L}^{\id_L}(\psi).$$  This proves that $\beta$ vanishes on coboundaries. \\

Since $m_6^\prime$ is a cocycle and Theorem~\ref{HHA} showed that $HH^6_{(-4)}(B,B)$ is one-dimensional, calculating $\beta(m_6^\prime)=-5 t^4e_4$ will complete the claim.  We use the relations in Proposition~\ref{Eisrelations} liberally and let $z=m_4\otimes 1^2-1\otimes m_4\otimes 1+1^2\otimes m_4$.  \\

We have
\begin{align*}
m_6(x) & =  (-M(2,0,0,1)-M(2,1,0,0)+M(3,0,0,0)+M(3,0,0,0)-\\ &M(1,0,2,0)+M(3,0,0,0)-M(2,1,0,0)-M(2,0,0,1))\id_L \\
& =  (-4M(2,1,0,0)+3M(3,0,0,0)-M(1,0,2,0))\id_L \\
& =  -2 t^4g_{2,1}\id_L, \\
f_3(z(x)) & =  M(1,0,0,0)g_{1,0}f_3(-\eta\theta\xi_L+\eta\theta\xi_L+(\xi_L)\eta\theta-  (\xi_L)\eta\theta+\eta\theta\xi_L+(\xi_L)\eta\theta-\\ &\eta\theta\xi_L-(\xi_L)\eta\theta-\eta\theta(\xi_L)+\eta\theta(\xi_L)-\eta(\xi)\theta-\eta(\xi)\theta+(\xi_L)\eta\theta-(\xi_L)\eta\theta) \\
& =  M(1,0,0,0)f_3(-2gef) \\
& =  -2[M(1,0,0,0)]^2\id_L, \\
-m_2(f_3\otimes f_3)(x) & =  -f_3((\xi_L)\eta\theta)f_3(\eta\theta\xi_L)+f_3(\eta\theta\xi_L)f_3((\xi_L)\eta\theta)+  f_3(\eta(\xi)\theta)f_3(\eta\theta\xi_L)+\\ &f_3((\xi_L)\eta\theta)f_3(\eta(\xi)\theta) \\
& =  0.
\end{align*}

So 
\begin{align*}
t_x^{\id_L}(m_6^\prime) & =  -2 t^4g_{2,1}-2 t^4g_{1,0}^2 \\
& =   t^4(2[e_2^*]^2-10e_4-2[e_2^*]^2)=-10 t^4e_4
\end{align*}
Also
\begin{align*}
m_6^\prime(\theta(\xi_L)\eta(\xi)\theta\xi_L) & = M(0,1,1,1)\theta-f_3(\theta(\xi_L)\eta)f_3(\xi\theta\xi_L) \\
& =   t^4(g_{2,1}+g_{1,0}^2)\theta=5e_4 t^4\theta,
\end{align*}
thus
$$t_{\theta(\xi_L)\eta(\xi)\theta\xi_L}^\theta(m_6^\prime)=5 t^4e_4$$ and $$\beta(m_6^\prime)=-10 t^4e_4+5 t^4e_4=-5 t^4e_4.$$

\end{proof}

\begin{prop} 
\label{m8prop}
\begin{enumerate} 
\item If $\phi\in \Hom_R^{(-6)}(B_+^{\otimes 8},B)$ is a cocycle,  then $t_{\eta(\xi)\theta(\xi_L)^2\eta(\xi)\theta}^{\id_L}(\phi)=0$ if and only if $\phi$ is a coboundary. \\
\item Let $c\in HH^8_{(-6)}(B,B)$ such that $t_{\eta(\xi)\theta(\xi_L)^2\eta(\xi)\theta}^{\id_L}(c)=-35$.  Then the $\C$-linear function $\gamma:HH^8_{(-6)}(B,B)\to\C$ defined by $c\mapsto 1$ is an isomorphism such that $m_8^\prime\mapsto  t^6e_6.$
\end{enumerate}\end{prop}

\begin{proof}
It is clear that for $\psi\in\Hom_R^{(-6)}(B_+^{\otimes 7},B)$ we have $$\delta\psi(\eta(\xi)\theta(\xi_L)^2\eta(\xi)\theta)=\eta\psi(\xi\theta(\xi_L)^2\eta(\xi)\theta)+\psi(\eta(\xi)\theta(\xi_L)^2\eta(\xi))\theta=0,$$ where both terms are zero for degree reasons.  Let 
\begin{align*}
\kappa_1&=f_3(m_6\otimes 1^2-1\otimes m_6\otimes 1+1^2\otimes m_6),\\
\kappa_2&=-m_6^\prime(\sum_{r+t=5}1^{\otimes r}\otimes f_3\otimes 1^{\otimes t}).\end{align*}  Then we have
\begin{align*}
m_8(\eta(\xi)\theta(\xi_L)^2\eta(\xi)\theta) & =  M(1,1,2,1)\id_L \\
& =  -\dfrac{1}{2} t^6g_{3,2}, \\
\kappa_1(\eta(\xi)\theta(\xi_L)^2\eta(\xi)\theta) & =   t^4f_3(M(0,1,2,0)\eta(\xi)\theta+M(0,2,1,0)\eta(\xi)\theta) \\
& =  2 t^6M(2,1,0,0)M(1,0,0,0) \\
& =  - t^6g_{2,1}g_{1,0}, \\
\kappa_2(\eta(\xi)\theta(\xi_L)^2\eta(\xi)\theta) & =  -M(1,0,0,0)m_6^\prime(\id_L(\xi_L)^2\eta(\xi)\theta-\eta\theta(\xi_L)\eta(\xi)\theta+\eta(\xi)\theta\eta(\xi)\theta-\\ &  \eta(\xi)\theta\eta(\xi)\theta-\eta(\xi)\theta(\xi_L)\eta\theta+\eta(\xi)\theta(\xi_L)^2\id_L) \\
& = M(1,0,0,0)m_6^\prime(\eta\theta(\xi_L)\eta(\xi)\theta+\eta(\xi)\theta(\xi_L)\eta\theta) \\
& =  M(1,0,0,0)[M(1,0,1,1)\id_L+M(1,1,1,0)\id_L+\\ &  M(1,0,0,0)f_3(gef+gef)-f_3(\eta\theta\xi_L)f_3(\eta(\xi)\theta)-f_3(\eta(\xi)\theta)f_3((\xi_L)\eta\theta) \\
& =  - t^6(2g_{2,1}g_{1,0}+2g_{1,0}^3+g_{1,0}^3-g_{1,0}^3)\id_L
\end{align*}

Thus
\begin{align*}
m_8^\prime(\eta(\xi)\theta(\xi_L)^2\eta(\xi)\theta) & =  t^6(-\dfrac{1}{2}g_{3,2}-g_{2,1}g_{1,0}-2g_{2,1}g_{1,0}-2g_{1,0}^3) \\
& =   t^6(-\dfrac{1}{2}g_{3,2}-3g_{2,1}g_{1,0}-2g_{1,0}^3) \\
& =  -35 t^6e_6.
\end{align*} \end{proof}

The vanishing of $HH^3_{(-2)}(B)$ implies that all choices of $f_3$ for the equivalence in Proposition~\ref{equiv} are homotopic, so that the cohomology classes of $m_6^\prime,m_8^\prime$ are well-defined in the following sense.

\begin{prop} Let $(f^\prime):A\to A$ be defined such that $f_1^\prime=\id_A$, $f_2^\prime=f_4^\prime=f_5^\prime=\cdots=0$, $f^\prime*m_4=f^\prime*m_{2k+1}=0$ for all $k\in\Z_{\geq 0}$.  Then
\begin{enumerate}
\item $m_6^\prime-f^\prime*m_6$ is a Hochschild coboundary, and \\
\item $m_8^\prime-f^\prime*m_8$ is a Hochschild coboundary.
\end{enumerate}\end{prop}

\begin{proof}
\begin{enumerate}
\item It follows from the discussion in section~\ref{HCAinf} (3) that $f_3-f_3^\prime$ is a cocycle, and since $HH^3_{(-2)}(B)=0$ it must also be a coboundary.  Let $\delta h=f_3-f_3^\prime$.  We have that
\begin{align*}
\zeta_1 &:=m_6^\prime-f^\prime*m_6 \\
&=\delta h(1^{\otimes 2}\otimes m_4-1\otimes m_4\otimes 1+m_4\otimes 1^{\otimes 2})-m_2(f_3\otimes f_3)+m_2(f_3^\prime\otimes f_3^\prime) \\
&=\delta h(1^{\otimes 2}\otimes m_4-1\otimes m_4\otimes 1+m_4\otimes 1^{\otimes 2})-m_2(f_3\otimes\delta h+\delta h\otimes f_3-\delta h\otimes \delta h).
\end{align*}
Let $\beta$ be as defined in Proposition~\ref{m6prop}.  The check that $\beta(\zeta_1)=0$ is a straightforward calculation.  We let $m_6^\prime-f^\prime*m_6=\delta g$.

\item Here we have
\begin{align*}
\zeta_2 &:= m_8^\prime-f^\prime*m_8 \\
&=\delta h(1^{\otimes 2}\otimes m_6-1\otimes m_6\otimes 1+m_6\otimes 1^{\otimes 2})-m_6^\prime\left(\sum_{r+1+t=6}1^{\otimes r}\otimes f_3\otimes 1^{\otimes t}\right)+ \\
&(f^\prime*m_6)\left(\sum_{r+1+t=6}1^{\otimes r}\otimes (f_3-\delta h)\otimes 1^{\otimes t}\right) \\
&=\delta h(1^{\otimes 2}\otimes m_6-1\otimes m_6\otimes 1+m_6\otimes 1^{\otimes 2})-\delta g\left(\sum_{r+1+t=6}1^{\otimes r}\otimes f_3\otimes 1^{\otimes t}\right)
\end{align*}
Let $\gamma$ be as defined in Proposition~\ref{m8prop}.  The check that $\gamma(\zeta_2)=0$ is a straightforward calculation.
\end{enumerate}

\end{proof}

\begin{theorem} Under the map $\alpha\oplus\gamma:HH^6_{(-4)}(B,B)\oplus HH^8_{(-6)}(B,B)\to \C^2$, we have $$(\alpha\oplus\gamma)(m_6^\prime,m_8^\prime)=(t^4e_4, t^6e_6).$$ \end{theorem}

We have set $C=\C/(\Z\oplus\Z\tau)$, whence $$j(\tau):=1728\dfrac{[e_4]^3}{[e_4]^3-27[e_6]^2}.$$  Since $t>0$, obviously we also have $$j=1728\dfrac{[\alpha(m_6^\prime)]^3}{[\alpha(m_6^\prime)]^3-27[\beta(m_8^\prime)]^2}.$$  In this way we recover the $j$-invariant of $C$ from this $A$-infinity structure. \\

\appendix
\section{Simplicial calculation}
\label{simpsec}

Let $k$ be a field.  Let $[n]=\{1,2,\ldots,n\}$.  We define a simplicial complex $\Delta[n]\subset P([n])$ such that $$\Delta_0=\{\{i\}|i\in [n]\}, \ \Delta_1=\{\{i,j\}|j-i\geq 2\}, \ \Delta_2=\{\{i,j,k\}|j-i\geq 2,k-j\geq 2\},$$ and in general $$\Delta_m=\{\{i_1,i_2,\ldots,i_m\}|i_{j+1}-i_j\geq 2, \ j=1,\ldots,m-1\}.$$

\begin{prop}
\label{simpprop} 
For all $k\in\N$, we have
\begin{align*}
\Delta[3k+1] & \simeq \textrm{point}, \\
\Delta[3k+2] & \simeq S^k, \\
\Delta[3k+3] & \simeq S^k.
\end{align*} \end{prop}

\begin{proof}  We proceed by induction on $k$.  The complex $\Delta[1]$ is literally a point, $\Delta[2]$ is two points and no edges, and $\Delta[3]$ is three vertices and the edge $\{1,3\}$. \\

Suppose the result for $k$.  Then $$\Delta[3k+4]=A\cup B$$ where \begin{align*}
A &=\{1,2,3,\ldots,3k+3\}\cap \Delta[3k+4], \\
B &=\{\textrm{all simplices containing the vertex $3k+4$}\}.
\end{align*}  The complex $B$ is contractible so $B\simeq D^k$ and $A\simeq \Delta[3k+3]\simeq S^k$.  Their intersection is $$A\cap B=\Delta[3k+3]\cap\{1,2,3,\ldots,3k+2\}\simeq \Delta[3k+2]\simeq S^k.$$  So $\Delta[3k+4]=D^k\cup S^k$ with $D^k\cap S^k=S^k$, which is contractible. \\

We proceed similarly for the other cases.  Now $\Delta[3k+5]=A\cup B$ where 
\begin{align*}
A &=\{1,2,\ldots,3k+4\}\cap\Delta[3k+5]\simeq \Delta[3k+4]\simeq D^{k+1},\\
B &=\{\textrm{all simplices containing 3k+5}\}\simeq D^{k+1}, \\
A\cap B &\simeq\Delta[3k+3]\simeq S^k.
\end{align*}
 So now we have two disks intersecting in $S^k$, which gives $S^{k+1}$.  \\ 

Finally $\Delta[3k+6]=A\cup B$ where 
\begin{align*}
A &=\{1,2,\ldots,3k+5\}\cap\Delta[3k+6]\simeq \Delta[3k+5]\simeq S^{k+1},\\
B &=\{\textrm{all simplices containing 3k+6}\}\simeq D^{k+1},\\
A\cap B &\simeq\Delta[3k+4]\simeq D^{k+1}.
\end{align*}
So we have an $S^{k+1}$ and a disk intersecting in a disk, which gives $S^{k+1}$. \end{proof}

For $k=0,1,2$, it will be helpful to have explicit representatives of the resulting homology class in $\Delta[3k+2]$ and $\Delta[3k+3]$.  For $\Delta[2]$ and $\Delta[3]$ we use $\{1\}-\{2\}$. \\

The loop in $\Delta[5]$ is constructed from gluing the contractible complex $\Delta[5]\cap\{1,2,3,4\}$ with the contractible complex of those simplices touching $\{5\}$.  The intersection is the $S^0$ in $\{1,2,3\}$.  The easiest way to realize this class is by taking the cone over $\{1\}\cup\{2\}$ to $\{4\}$ and another cone to $\{5\}$.  Thus the resulting loop is $\{1,5\}\pm\{1,4\}\pm\{2,5\}\pm\{2,4\}$.  We also use this class in $\Delta[6]$.  (The choice between + and $-$ when offered is not of importance for our application of this calculation, so we do not make it.) \\

The class in $\Delta[8]$ we realize similarly.  The intersection of the two contractible parts is $\Delta[6]$, which we consider as the loop above.  We make a cone over this loop to the points $\{7\}$ and $\{8\}$ to get the class representative $$\{1,5,7\}\pm\{1,4,7\}\pm\{2,5,7\}\pm\{2,4,7\}\pm\{1,5,8\}\pm\{1,4,8\}\pm\{2,5,8\}\pm\{2,4,8\}.$$ This class will also work for $\Delta[9].$ \\

\end{document}